\documentclass[11pt,reqno]{article}
\usepackage{amsmath}
\usepackage{amsthm}
\usepackage{amssymb}
\usepackage{amsfonts}
\usepackage{cases}
\usepackage{graphicx}
\usepackage{xcolor}
\usepackage[margin=1in]{geometry}
\usepackage[utf8]{inputenc}
\usepackage[margin=1in]{geometry}
\usepackage{verbatim}
\usepackage{caption,subcaption}
\usepackage{algorithm}
\usepackage{algpseudocode}
\usepackage[normalem]{ulem}
\makeatletter
\renewcommand{\fnum@algorithm}{\fname@algorithm}
\makeatother
\usepackage[colorlinks,citecolor=blue,urlcolor=blue]{hyperref}
\usepackage{bm}

\usepackage{comment}
\usepackage{stmaryrd}

\numberwithin{equation}{section}

\newtheorem{Remark}{Remark}[section]
\newtheorem{Theorem}{Theorem}[section]
\newtheorem{Lemma}{Lemma}[section]
\newtheorem{Proposition}{Proposition}[section]
\newtheorem{Corollary}{Corollary}[section]

\newcommand{\be}{\begin{equation}}
	\newcommand{\ee}{\end{equation}}
\newcommand{\bee}{\begin{equation*}}
	\newcommand{\eee}{\end{equation*}}
\newcommand{\bi}{\begin{itemize}}
	\newcommand{\ei}{\end{itemize}}
\DeclareMathOperator{\tr}{tr}
\DeclareMathOperator*{\argmax}{arg\,max}

\usepackage{algorithm}
\usepackage{algpseudocode}
\usepackage{soul}

\def \E{\mathbb{E}}

\def \N{\mathbb{N}}

\def \R{\mathbb{R}}

\def \Cc{{\mathcal C}}

\def \Ec{{\mathcal E}}
\def \Pc{{\mathcal P}}

\def \eps{\varepsilon}

\def \Hc{\mathcal{H}}

\def \E{\mathbb{E}}

\def \N{\mathbb{N}}
\def \I{\mathbb{I}}

\def \R{\mathbb{R}}

\def \Cc{{\mathcal C}}

\def \Ec{{\mathcal E}}
\def \Pc{{\mathcal P}}

\def \Sc{{\mathcal S}}
\def \eps{\varepsilon}

\def \Hc{\mathcal{H}}

\def \bbR{\mathbb{R}}

\newcommand{\beq}{\begin{equation}}
\newcommand{\eeq}{\end{equation}}
\newcommand{\lb}{\label}

\title{Policy Iteration for Exploratory Hamilton--Jacobi--Bellman Equations}

\date{}	

\author
{Hung Vinh Tran\thanks{Department of Mathematics, University of Wisconsin Madison, Van Vleck Hall, 480 Lincoln Drive, Madison, WI 53706, \texttt{hung@math.wisc.edu}.}
\and Zhenhua Wang\thanks{Corresponding author. Zhongtai Securities Institute for Financial Studies, Shandong University, 27 Shanda Nanlu, Jinan, Shandong, 250100, \texttt{zhenhuaw@sdu.edu.cn}.}
\and Yuming Paul Zhang\thanks{Department of Mathematics and Statistics, Auburn University, Parker Hall, 221 Roosevelt Concourse, Auburn, AL 36849, \texttt{yzhangpaul@auburn.edu}.}
}

\begin{document}
\maketitle

\begin{abstract}
 We study the policy iteration algorithm (PIA) for entropy-regularized stochastic control problems on an infinite time horizon with a large discount rate, focusing on two main scenarios.  First, we analyze PIA with bounded coefficients where the controls applied to the diffusion term satisfy a smallness condition. We demonstrate the convergence of PIA based on a uniform $\Cc^{2,\alpha}$ estimate for the value sequence generated by PIA, and provide a quantitative convergence analysis for this scenario. Second, we investigate PIA with unbounded coefficients but no control over the diffusion term. In this scenario, we first provide the well-posedness of the exploratory Hamilton--Jacobi--Bellman equation with linear growth coefficients and polynomial growth reward function. By such a well-posedess result we achieve PIA's convergence by establishing a quantitative locally uniform $\Cc^{1,\alpha}$ estimates for the generated value sequence.     
\end{abstract}

\maketitle

\textit{Key words:} Hamilton--Jacobi--Bellman equations, policy iteration algorithm, stochastic control, reinforcement learning, entropy regularization, convergence rate.

\medskip

\textit{AMS subject classifications:}
  35F21, 60J60, 68W40, 93E20.

\section{Introduction}
Policy improvement involves updating the current strategy to enhance performance. 
This iterative process, known as a policy improvement algorithm (PIA, also called policy iteration algorithm), aims at converging towards an optimal policy through successive refinements. Rooted in Dynamic Programming, PIA plays an important role in Markov decision processes problems (MDPs) and reinforcement learning (RL), which could be dated back to Bellman \cite{bellman1957dynamic} and Howard \cite{howard1960dynamic}. 
The convergence rate of PIA for an infinite time horizon was investigated in Puterman, Brumelle \cite{puterman1979convergence}. For discrete-time MDPs, PIA has been well explored and its convergence has been established under suitable conditions on the model parameters, see e.g., Puterman \cite{puterman2014markov}, Sutton, Barto \cite{sutton2018reinforcement}, and Bertsekas \cite{bertsekas2015value}, among many others.

In the framework of controlled ODEs or deterministic optimal control, the convergence of PIA has been studied in linear quadratic settings, see Abu-Khalaf, Lewis \cite{abu2005nearly}, Vrabie, Pastravanu, Abu-Khalaf,  Lewis \cite{vrabie2009adaptive}, and the references therein. Lee and Sutton \cite{lee2021policy} proved the convergence under certain regularity and fixed point assumptions. To overcome the ill-posedness of PIA for general controlled ODEs,  Tang, Tran, and Zhang \cite{tang2023policy} proposed a semi-discrete scheme for the PIA and showed its general exponential convergence rate.
Lee and Kim \cite{LeeKim} incorporated a deep operator network with the scheme in \cite{tang2023policy} to numerically solve the PIA and the optimal control problems.

For stochastic control problems (without entropy-regularization), 
Krylov \cite{krylov2008controlled} and Puterman \cite{Puterman81} showed that the optimal value is recovered under the PIA for a specific control problem with a compact space-time domain.
Jacka and Mijatovi\'c \cite{jacka2017policy} outlined a list of assumptions leading to PIA's convergence towards optimal control, offering illustrative examples. 
Afterward, Kerimkulov, \v{S}i\v{s}ka, and Szpruch \cite{kerimkulov2020exponential} established the convergence and studied the convergence rate by assuming a certain regularity of the optimal value function and the control does not appear in the diffusion term. 
The convergence of PIA was further studied in some problems in mean-field games, see Cacace, Camilli, Goffi \cite{cacace2020policy}, and Camilli, Tang \cite{camilli2021rates}.

In the RL literature, it is now well-known that the entropy regularization (also termed the ``softmax" criterion) encourages exploration of the unknown environment through measure-valued control strategies (i.e., relaxed controls). This approach prevents early settlement to suboptimal strategies, a problem known as the curse of optimality, for a detailed explanation, see Zhou \cite{zhou2021curse}. With entropy regularization, Wang, Zariphopoulou, and Zhou \cite{wang2019exploration} opened the door to incorporating continuous-time stochastic control problems into the exploratory framework of RL. The following so-called {\it exploratory} Hamilton-Jacobi-Bellman (HJB) equation thereby was introduced and first studied in a linear-quadratic setting in \cite{wang2019exploration}.
\be\label{eq:HJB0}
\begin{aligned}
\rho v-\sup_{\pi\in \Pc(U)} \left\{ \int_U \left[ b(x,u)\cdot {D}  v+\frac{1}{2} \tr(\sigma\sigma^T(x,u) D^2 v(x))+r(x,u)-\lambda \ln(\pi(u)) \right] \pi(u)\,du  \right\} = 0.
\end{aligned}
\ee
Here, $\Pc(U)$ represents all probability densities on the action space $U$.
Moreover, as the weight $\lambda$ tends to zero, the optima and optimal relaxed strategy for an exploratory stochastic control problem converge to the optima and optimal feedback control for the corresponding standard stochastic control problem, respectively, as proved by Tang, Zhang, and Zhou \cite{tang2022exploratory}.
Such entropy-regularization formulation has been further extended to various settings, such as mean-variance problems (Wang, Zhou \cite{wang2020continuous}), stopping problems (Dong \cite{dong2022randomized}), and mean-field games/mean-field controls (Guo, Xu, Zariphopoulou \cite{Guo22}, Firoozi, Jaimungal \cite{MR4385154}, Frikha, Germain, Lauri{\`e}re, Pham, Song \cite{frikha2023actor}, Wei, Yu \cite{wei2023continuous}).

The PIA for the HJB equation \eqref{eq:HJB0} can be written as follows in a PDE framework.
\begin{algorithm}
	\caption{Policy Iteration Algorithm for \eqref{eq:HJB0}}
        \label{alg:PIA}
	\begin{algorithmic}
		\State {\bf Initialization:} Take suitable $v^0\in \Cc^{2}$. \State{\bf Iteration:} For $n\in \N$, {\bf do}
		\begin{align}
		&\text{Set}\;	\pi^n(x,u) :=\frac{\exp\left[\frac{1}{\lambda }\left(  b(x,u)\cdot {D}  v^{n-1}(x)+2^{-1}\tr(\sigma\sigma^T(x,u) D^2 v^{n-1}(x))+r(x,u) \right)\right]}{\int_U\exp\left[ \frac{1}{\lambda }\left( b(x,u')\cdot {D}  v^{n-1}(x)+2^{-1}\tr(\sigma\sigma^T(x,u') D^2 v^{n-1}(x))+r(x,u') \right)\right] \,du'}; \label{eq:PIApin}\\
		& \text{Solve}\; v^n(x) \; \text{for} \notag
		\end{align}
	\be\label{eq:PIAvn}  
	\rho v^n-\int_U \left[ b(x,u)\cdot {D}  v^n+\frac{1}{2} \tr(\sigma\sigma^T(x,u) D^2 v^n(x))+r(x,u)-\lambda \ln(\pi^n(u)) \right] \pi^n(u)\,du= 0.
	\ee
	\end{algorithmic}
\end{algorithm}


For entropy-regularized PIA, \cite{wang2020continuous} demonstrated its policy improvement property for linear-quadratic mean-variance problems and shows that optimal solutions can be achieved within two iterations due to the Gaussian densities of the optimal feedback control in this setting. \cite{dong2022randomized} introduced the PIA for entropy-regularized stopping problems under a finite horizon and proved its convergence, where the underlying process is a geometric Brownian motion. Notably, the policy improvement property may fail when seeking Nash equilibrium strategies in a time-inconsistent stochastic control problem, as discussed in Dai, Dong, and Jia \cite{dai2023learning}. 
Huang, Wang, and Zhou \cite{pia2022} proved a qualitative convergence result with bounded coefficients when the diffusion term is not controlled (that is, when $\sigma$ is independent of $u$). In a very recent paper by Tang and Zhou \cite{tang-zhou2024}, a regret analysis in terms of the iteration step $n$ and the entropy weight $\lambda$ was provided under similar conditions in a finite horizon setting.

Despite these developments, the general convergence of Algorithm \ref{alg:PIA} remains largely unexplored, especially when the coefficients $b,\sigma,r$ are unbounded and controls appear in the diffusion term. The first question concerns the well-posedness of PIA: whether the updated control $\pi^n$ in \eqref{eq:PIApin} is well-defined and whether the elliptic equation \eqref{eq:PIAvn} admits a unique solution. Once this is resolved, the next goal is to investigate the convergence of the generated sequence $\{v^n\}_n$ to the optimal solution of the entropy-regularized stochastic control, which should solve the HJB equation \eqref{eq:HJB0}. As indicated in \eqref{eq:PIApin}, entropy regularization makes the policy sequence $\{\pi^n\}_n$ less singular, but it introduces second-order derivatives of the value sequences into the entropy of the policy sequence. This, in turn, makes achieving the compactness of the policy sequence more challenging.

\medskip

In this paper, we investigate PIA with sufficiently large discounting rates in two scenarios. 
Firstly, we study PIA with bounded coefficients, allowing the control to influence the diffusion terms. This generalization is significant and challenging, as most literature on PIA considers problems with controls in the source or drift terms, resulting in linear or semilinear partial differential equations. Specifically, in their iterations, the updated policies are independent of the second-order derivatives of the solutions. In contrast, in our setting, the limiting equation is fully nonlinear, and the iterated controls depend also on the second derivatives.

So far, we have proved the convergence of PIA under the condition that the control has a small effect on the diffusion. Here, the smallness condition refers to the controlled volatility of the diffusion term being close to an uncontrolled volatility, see \eqref{cond2} for the precise formulation. The key step is to obtain a uniform $\Cc^{2,\alpha}$ estimate for the solutions in the iteration (see Theorem \ref{T.1.1}). This result can be viewed as a version of the Evans-Krylov theorem for PIA. The smallness assumption is essential in the proof of the uniform regularity result, as there is no obvious regularizing effect from the iteration process. It is unclear whether PIA converges without the smallness assumption and regardless of uniform regularity. Furthermore, we provide quantitative results on the convergence in Theorem \ref{T.2.1} and Remark \ref{rem:quant}.
In particular, Remark \ref{rem:quant} gives an exponential convergence rate for the result in \cite{pia2022} when the diffusion term is not controlled.

The second objective is to investigate PIA with unbounded coefficients in $\R^d$, where the controls appear only in lower order terms. In this scenario, we allow the coefficients $b$ and $\sigma$ to exhibit linear growth in $x$, and the reward function $r$ to have polynomial growth, encompassing linear quadratic settings. It is important to note that under such settings, the well-posedness of solutions to elliptic equations or PIA is not guaranteed without additional conditions. Specifically, $\rho$ needs to be large, and our results are optimal in terms of the growth rate of $r$ (see Proposition \ref{P.4.1} and Remark \ref{rm1}). To show the convergence of PIA, we obtain locally uniform quantitative $\Cc^{1,\alpha}$ estimates for the solutions. Here, the analysis is localized, and the estimates must be done carefully with explicit dependence on the size of the coefficients. 
In particular, with the help of the logarithmic growth of the entropy provided in \cite[Corollary 4.2]{pia2022}, we also recover the convergence result from that work in Theorem \ref{T.5.1}.

We also comment that our approach works well for the finite horizon problem. The uniform $\Cc^{2,\alpha}$ regularity follows from the combination of our argument and  \cite[Theorems 6.4.3 and 6.4.4]{Krybook}.

After our paper was finished, we learned that Ma, Wang, and Zhang \cite{MWZ24} obtained some very interesting related results independently using a different approach when the diffusion term is not controlled. 
Specifically, the authors employed Feynman-Kac type probabilistic representation formulas for solutions of the iterative PDEs and their derivatives.
By this method, they first provided a simple proof of the convergence result in \cite{pia2022}, and then achieved an exponential convergence rate in the infinite horizon model with a large discount factor (which is essentially the same as Remark \ref{rem:quant}) and in the finite horizon model.


\subsection{Model formulation}
Denote by $|\cdot|$ the Euclidean norm. For a fixed dimension $d\in\N$, we denote by {$B_R(x)$} the open ball centered at $x\in\R^d$ with radius $R$, and we simply write $B_R$ for $B_R(0)$. We further use $\Sc^d$ to denote the set of all symmetric, non-negative definite matrices of size $d$. We use $\I_d$ to denote the identity matrix of size $d$.

For a multi-index $a=(a_1,\cdots, a_m)$, denote $|a|_1 = \sum_{i=1}^m a_i$.
Given a non-empty connected open set $V\subseteq \R^d$, for $k\in\N\cup\{0\}$ and $\alpha\in (0,1)$, we define
\[
[w]_{\Cc^{k,\alpha}(V)}:= \sup_{x,y\in V, |x-y|\leq 1, |a|_1=k}\frac{|D^a w(x)-D^a w(y)|}{|x-y|^\alpha}, \quad \|w\|_{L^\infty(V)}:= \sup_{x\in V}|w(x)|,
\]
\[
\llbracket w\rrbracket_{\Cc^k(V)}:=\sup_{|a|_1=k}\|D^a w\|_{L^\infty(V)},\quad \|w\|_{\Cc^{k,\alpha} (V)}:=\sum_{0\leq j\leq k} \left(  \llbracket w\rrbracket_{\Cc^j (V)}+[w]_{\Cc^{j,\alpha }(V)} \right),\quad
\]
and
\[
 \|w\|_{\Cc^k(V)}:=\sum_{0\leq j \leq k} \llbracket w\rrbracket_{\Cc^j(V)},
\]
and denote by $\Cc^{k,\alpha}(V)$ the set of all $w$ such that $\|w\|_{\Cc^{k,\alpha}(V)}<\infty$. For any vector-valued or matrix-valued function $w=(w_{ij}(x))$, we denote
\[
|w(x)|:=\max_{ij}|w_{ij}(x)|\quad\text{and}\quad \|w\|_{L^\infty(V)}:=\sup_{x\in V}|w(x)|.
\]
When $V$ equals to the whole space $\R^d$, we simply write 
\[
[w]_{k,\alpha}=[w]_{\Cc^{k,\alpha}(\R^d)},\quad \|w\|_{\infty}=\|w\|_{L^\infty(\R^d)}, \quad \llbracket w\rrbracket_{k}=\llbracket w\rrbracket_{\Cc^k(\R^d)}, \]
\[
\|w\|_{k,\alpha}=\|w\|_{\Cc^{k,\alpha}(\R^d)},\quad\|w\|_{k}=\|w\|_{\Cc^k(\R^d)},
\]
For any $q\ge 1$, we will also use the Sobolev norm $\|\cdot\|_{W^{k,q}(V)}$.

\medskip

Fix positive integers \( d \), \( m \), and \( l \). Consider an \( m \)-dimensional Brownian motion and \(\mathcal{F} = (\mathcal{F}_t)_{t \geq 0}\) as the natural filtration generated by \((W_t)_{t \geq 0}\), that is, \(\mathcal{F}_t = \sigma(W_s, 0 \leq s \leq t)\) for all \( t \geq 0 \). Let \( U \) be a bounded domain in \( \mathbb{R}^l \) with its Lebesgue measure \(|U|\in (0,\infty)\). Without loss of generality, we simply assume $|U|=1$ throughout the paper.
Denote by \(\mathcal{P}(U)\) the set of all probability density functions on \( U \). Given a relaxed control \(\pi = (\pi_t)_{t \geq 0}\), which is \(\mathcal{F}\)-adapted and \(\pi_t \in \mathcal{P}(U)\) for all \( t \geq 0 \), we consider the controlled process
\[
dX^\pi_t = \left(\int_U b(X^\pi_t,u)\pi_t(u)\,du\right) dt+\sqrt{\left(\int_U \sigma (X^\pi_t,u)\sigma (X^\pi_t,u)^T\pi_t(u)\,du\right)}\,dW_t,
\]
with $X_0=x\in \R^d$ and functions $b(x,u)=(b_1(x,u),...,b_d(x,u)): \R^d\times U\to \R^d$, $\sigma(x)= (\sigma_{ij}(x))_{i, j}:\R^d\to \R^{d\times m}$. The value function under $\pi$ is given by
$$
V^\pi(x):=\E_x\left[\int_0^\infty {e^{-\rho t}} \left( \int_U r(X^\pi,u)\pi_t(u)\,du-\lambda \int_U \ln(\pi_t(u)) \pi_t(u)\,du\right) \,dt \right].
$$
For any $x, p\in \R^d, X\in \Sc^d$, and $\pi\in \Pc(U)$, define 
\be\label{eq:Fpi} 
F_\pi(x, p, X):=  \int_U \left[r(x,u)+ b(x,u) \cdot p+ \frac{1}{2}\tr(\sigma\sigma^T(x,u) X) -\lambda \ln(\pi(u))\right] \pi(u)\,du,
\ee 
and
\beq\lb{0.2}
F(x,p,X):=\sup_{\pi\in \Pc(U)}F_\pi(x,p,X),
\eeq
By Dynamic Programming argument, $v^*(x):=\sup_{\pi}V^\pi(x)$ is a viscosity solution to the HJB equation \eqref{eq:HJB0}, which is now rewritten as
\be\label{eq:HJB} 
\rho v-F(x,Dv,D^2v)=0
\ee

Given $x, p\in \R^d, X\in \Sc^d$, set 
\beq\lb{0.1}
\Gamma(x, p, X)(u):=\frac{\exp\left[\frac{1}{\lambda}\left(   r(x,u)+ b(x,u) \cdot p+ \frac{1}{2}\tr(\sigma\sigma^T(x,u) X)\right)\right]}{\int_U \exp\left[  \frac{1}{\lambda} \left(r(x,u')+ b(x,u') \cdot p+ \frac{1}{2}\tr(\sigma\sigma^T(x,u') X)\right) \right]\,du'}.
\eeq
We claim that $\pi(u):=\Gamma(x,p,X)(u)$ is the maximizer of \eqref{0.2} (see the variational principle in \cite[Proposition 1.4.2]{dupuis1997weak}). Indeed, let $\phi(u)$ be such that $\int_U\phi(u)du=0$ and $\pi+\eps \phi\geq 0$ for all $\eps\in (-1,1)$. Then $\eps=0$ is a critical point of $G(\pi+\eps\phi):=F_{\pi+\eps\phi}(x,p,X)$, and the first variation of $G$ satisfies
\[
\int_U \left[r(x,u)+ b(x,u) \cdot p+ \frac{1}{2}\tr(\sigma\sigma^T(x,u) X) -\lambda \ln(\pi(u))\right] \phi(u)\,du=0.
\]
This holds for all such $\phi$, so $r(x,u)+ b(x,u) \cdot p+ \frac{1}{2}\tr(\sigma\sigma^T(x,u) X) -\lambda \ln(\pi(u))$ is constant in $u$. Using that $\int_U \pi(u)du=1$ yields the formula \eqref{0.1}.

In particular, we have
\beq\lb{0.3}
F_{\pi^*}(x,Dv^*,D^2 v^*)=F(x,Dv^*,D^2 v^*)\quad \text{ with }\pi^*(x,u):=\Gamma(x,Dv^*(x),D^2v^*(x))(u).
\eeq
Recall the PIA, and $\pi^n$ defined in \eqref{eq:PIApin} can be rewritten as
\beq\lb{1.11}
\begin{aligned}
\pi^n(x,u)=&\Gamma(x,Dv^{n-1}(x),D^2v^{n-1}(x))(u)\\
=&
\argmax_{\pi(\cdot)\in \Pc(U)}F_\pi(x,Dv^{n-1}(x),D^2v^{n-1}(x)),
\end{aligned}
\eeq
and \eqref{eq:PIAvn} becomes
\begin{equation}\label{eq:PI}
    \rho v^n-F_{\pi^{n}}(x,Dv^n, D^2v^{n})=0.
\end{equation}

We will require $\rho$ to be large in both the bounded and unbounded settings, which is essential for a better regularity of the solutions.
This is known in the literature, and we refer the reader to \cite[Theorem 2.8]{tranbook}, which discusses the first-order case.

\subsection{Organization of the paper}
The paper is organized as follows. Sections \ref{sec:bounded.sigma} and \ref{sec: converge.bounded} are dedicated to the first objective for bounded coefficients where the controls applied to the diffusion term satisfy a smallness condition. In Section \ref{sec:bounded.sigma}, Theorem \ref{T.1.1} demonstrates the locally uniform $\Cc^{2,\alpha}$ bound for the sequence $\{v^n\}_n$ generated by PIA. Section \ref{sec: converge.bounded} establishes the convergence of PIA in Theorem \ref{thm:bound.converge} and provides a quantitative result in Theorem \ref{T.2.1} for this scenario.
Section \ref{sec:unbound.nosigma} studies the well-posedness of the nonlinear elliptic equation \eqref{eq:HJB} for the case of unbounded coefficients. 
Section \ref{subsec:unbound.convergence} then addresses the well-posedness of PIA and its convergence for the second scenario with unbounded coefficients where the diffusion term is independent of the control.

\section{Uniform $\Cc^{2,\alpha}$ estimates for bounded equations}\label{sec:bounded.sigma}


We first estimate several norms of solutions with explicit dependence on $\rho\geq 1$. 
\begin{Lemma}\lb{L.1}
Let $\rho\geq 1$, and let $v$ be a solution to
\[
\rho v(x)- \tilde{r}(x) - \tilde{b}(x) \cdot {D}  v(x)-2^{-1}\tr(\tilde\Sigma(x) {D} ^2 v(x)) =0.
\]
Assume that $\tilde\Sigma\geq \I_d/C_0$ for some $C_0>0$. Then there exists an increasing function $\eta: [1,\infty)\to [1,\infty)$ independent of $\rho$ such that if
\beq\lb{2.1'}
\|\tilde r(\cdot)\|_{0,\alpha},\,\|\tilde b(\cdot)\|_{0,\alpha} ,\,\|\tilde\Sigma(\cdot)\|_{0,\alpha}\leq A\quad \text{ for some $A\geq 1$ and $\alpha\in (0,1)$,}
\eeq
we have
\beq\lb{1.1}
\rho\|v\|_{\infty},\;\rho^{1-\frac\alpha2}[v]_{0,\alpha},\; \rho^{\frac12-\frac\alpha2}[v]_{1,\alpha},\; \rho^{-\frac\alpha2}[v]_{2,\alpha},\; \rho^\frac12\llbracket v\rrbracket_1,\; \llbracket v\rrbracket_2\leq \eta(A).
\eeq
\end{Lemma}

\begin{proof}
By comparing $v$ with $\pm\|\tilde{r}\|_{\infty}/\rho$, we obtain
$\|v\|_{\infty}\leq \|\tilde{r}\|_{\infty}/{\rho}$. 

To prove \eqref{1.1}, we apply a scaling argument.
Note that $w(x):={v}(x/\sqrt{\rho})$ solves
\[
w- \frac{\tilde{r}(x/\sqrt{\rho})}\rho - \frac{\tilde{b}(x/\sqrt{\rho})}{\sqrt{\rho}} \cdot {D}  w-\frac{1}{2}\tr(\tilde{\Sigma}(x/\sqrt{\rho}) {D} ^2w) =0.
\]
It is direct to see that the $\alpha$-H\"{o}lder norms of 
\[
\frac{\tilde{r}(x/\sqrt{\rho})}\rho,\quad \frac{\tilde{b}(x/\sqrt{\rho})}{\sqrt{\rho}},\quad \tilde{\Sigma}(x/\sqrt{\rho}) 
\]
are non-increasing as $\rho\geq 1$ increases, and also that $\tilde{\Sigma}(x/\sqrt{\rho})\geq \I_d/C_0$ is preserved.
Thus, the famous interior Schauder estimates state
\[
\|w\|_{2,\alpha}\leq C_A (\|w\|_{\infty}+\rho^{-1}\| \tilde{r}(\cdot/\sqrt{\rho})\|_{0,\alpha})\leq C_A/\rho
\]
where $C_A$ depends only on $A$ and the dimension.
This implies that
\[
 [v]_{0,\alpha} \leq C_A\rho^{-1+\frac\alpha2},\quad [v]_{1,\alpha}\leq C_A\rho^{-\frac12+\frac\alpha2},\quad [v]_{2,\alpha}\leq C_A\rho^{\frac\alpha2},\quad \llbracket v\rrbracket_1\leq C_A\rho^{-\frac12},\quad \llbracket v\rrbracket_2\leq C_A.
\]
\end{proof}

\begin{Remark}
    We note that the bound \eqref{1.1} is essentially optimal thanks to the scaling approach.
    In particular, the bound $[v]_{2,\alpha}\leq C_A\rho^{\frac\alpha2}$ cannot be improved in general.
    Here is another way to see it.
    Indeed, \eqref{1.1} gives $\rho \|v\|_{0,\alpha} \leq C_A \rho^{\frac{\alpha}{2}}$.
    Hence, by applying the interior Schauder estimates directly to $v$, we get
    \[
    \|v\|_{2,\alpha}\leq C_A \left(\rho\|v\|_{0,\alpha}+\| \tilde{r}\|_{0,\alpha}\right)\leq C_A \rho^{\frac{\alpha}{2}}.
    \]
\end{Remark}

Recall $\Gamma(x, p, X)(u)$ from \eqref{0.1},
and define
\[
\pi(x,u):=\Gamma(x, Dv,D^2v)(u).
\]
Denoting $\Sigma=\sigma\sigma^T$, we assume that there exists $C_0\geq 1$ such that uniformly for all $u\in U$,
\beq\lb{cond1}
\Sigma(\cdot,u)\geq \I_d/C_0,\quad \|r(\cdot,u)\|_{0,\alpha},\,\|b(\cdot,u)\|_{0,\alpha} ,\,\|\Sigma(\cdot,u)\|_{0,\alpha}\leq C_0.
\eeq
We have the following estimates for $\pi$ and the entropy.

\begin{Lemma}\lb{L.2}
Let $v$ satisfy \eqref{1.1} with $A_1$ in place of $\eta(A)$. Assume \eqref{cond1}, $|U|=1$, $\rho\geq 1$, and that there exist $\Sigma_0(x)$ and $\eps_0,\eps_1\in (0,1)$ such that for $M(x,u):=\Sigma(x,u)-\Sigma_0(x)$,
\beq\lb{cond2}
\sup_{u\in U} \|M(\cdot,u)\|_{\infty}\leq \eps_0,\quad \sup_{u\in U}[M(\cdot,u)]_{0,\alpha}\leq \eps_1.
\eeq
Then there exists $C$ depending only on $\lambda,C_0$ such that 
$
\pi(x,u)=\Gamma(x,Dv,D^2 v)(u)$
satisfies 
\begin{align*}
&\|\pi(\cdot,u)\|_{\infty}\leq\exp\left[C(1+A_1(\rho^{-\frac12}+\eps_0 ))\right],
\\
&[\pi(\cdot,u)]_{0,\alpha}\leq \left(1+A_1(\rho^{\frac\alpha2-\frac12}+\eps_1+\rho^\frac{\alpha}{2}\eps_0)\right)\exp\left[C(1+A_1(\rho^{-\frac12}+\eps_0 ))\right],
\\ 
&\left\|\int_U \pi(\cdot,u)\ln\pi(\cdot,u)\,du\right\|_{\infty}\leq C\left(1+A_1(\rho^{-\frac12}+\eps_0)\right),
\\
&\left[\int_U\pi(\cdot,u)\ln \pi(\cdot,u)\,du\right]_{0,\alpha}\leq \exp\left[C\left(1+A_1(\rho^{\frac\alpha2-\frac12}+\eps_1+\rho^\frac{\alpha}{2}\eps_0) \right)\right].  
\end{align*}
\end{Lemma}
\begin{proof}
Let us write 
\[
f(x,u)=\frac{1}{\lambda}\left[r(x,u)+ b(x,u) \cdot Dv+ \frac{1}{2}\tr(\Sigma(x,u) D^2v)\right],
\]
and
\[
g(x,u)=\frac{1}{\lambda}\left[r(x,u)+ b(x,u) \cdot Dv+ \frac{1}{2}\tr((\Sigma(x,u) -\Sigma_0(x))D^2v)\right].
\]
It is easy to see that
\beq\lb{1.4}
\pi(x,u)=\frac{\exp(f(x,u))}{\int_U\exp(f(x,u'))\,du'}=\frac{\exp(g(x,u))}{\int_U\exp(g(x,u'))\,du'}.
\eeq

By the assumptions of \eqref{cond1}, \eqref{cond2}, and that \eqref{1.1} holds with $A_1$ in place of $\eta(A)$, it follows that for some $C=C(C_0,\lambda)$,
\beq\lb{1.2}
|g(x,u)|\leq C(1+\rho^{-\frac12}A_1+\eps_0 A_1).
\eeq
Therefore, also using $|U|=1$, we get
\beq\lb{1.3}
\exp\left[-C(1+\rho^{-\frac12}A_1+\eps_0 A_1)\right]\leq \pi(x,u)\leq \exp\left[C(1+\rho^{-\frac12}A_1+\eps_0 A_1)\right]
\eeq

Next, note that
\begin{align*}
&\left|\pi(x,u)-\pi(y,u)\right|=\left|\frac{\exp(g(x,u))\int_U\exp(g(y,u'))\,du'-\exp(g(y,u))\int_U\exp(g(x,u'))\,du'}{\int_U\exp(g(x,u'))\,du'\int_U\exp(g(y,u'))\,du'}\right|  \\
&\qquad\quad \leq \left|\frac{\int_U\exp(g(y,u')-g(x,u'))\,du'}{\int_U\exp(g(y,u'))\,du'}\right| \pi(x,u)+\left|\frac{\exp(g(x,u))-\exp(g(y,u))}{\int_U\exp(g(y,u'))\,du'}\right|.
\end{align*}
Therefore, \eqref{1.1}, \eqref{cond1}, and \eqref{cond2} yield
\beq\lb{1.5}
\begin{aligned}
[g(\cdot,u)]_{0,\alpha}&\leq {\lambda^{-1}}\left([r]_{0,\alpha}+[b\cdot Dv]_{0,\alpha}+[\Sigma-\Sigma_0]_{0,\alpha}\llbracket v \rrbracket_{2}+\|\Sigma-\Sigma_0\|_{\infty}[v]_{2,\alpha}\right)\\
&\leq C\left(1+A_1\rho^{\frac\alpha2-\frac12}+A_1\eps_1+A_1\rho^{\frac\alpha2}\eps_0  \right),
\end{aligned}
\eeq
and
\[
[\exp g(\cdot,u)]_{0,\alpha}\leq \left(1+A_1\rho^{\frac\alpha2-\frac12}+A_1\eps_1+A_1\rho^{\frac\alpha2}\eps_0 \right)\exp\left[C(1+\rho^{-\frac12}A_1+\eps_0 A_1)\right].
\]
Also using \eqref{1.2} and \eqref{1.3}, we obtain
\[
[\pi(\cdot,u)]_{0,\alpha}\leq \left[1+A_1(\rho^{\frac\alpha2-\frac12}+\eps_1+\rho^\frac{\alpha}{2}\eps_0)\right]\exp\left[C(1+A_1(\rho^{-\frac12}+\eps_0 ))\right].
\]
By \eqref{1.3} again and that $\pi$ is a probability distribution, we get
\[
\int_U |\ln \pi(x,u)|\pi(x,u)\,du\leq \|\ln \pi(\cdot,\cdot) \|_{\infty}\leq C\left[1+A_1(\rho^{-\frac12}+\eps_0)\right].
\]

For the last claim, note that
\begin{align*}
&\left|\int_U \pi(x,u)\ln \pi(x,u)\,du-\int_U \pi(y,u)\ln \pi(y,u)\,du \right|\\
&\qquad\qquad\leq \sup_{u\in U}|\ln \pi(x,u)-\ln\pi(y,u)|+\sup_{u\in U} \left|\frac{\pi(y,u)}{\pi(x,u)}-1\right||\ln \pi(y,u)|.    
\end{align*}
So let us estimate the right-hand side two terms below.
Because of \eqref{1.4} and \eqref{1.5}, we get 
\begin{align*}
|\ln \pi(x,u)-\ln\pi(y,u)|&=\left|g(x,u)-g(y,u)-\ln \int_U e^{g(x,u')-g(y,u')} \,du'\right|\\
&\leq C\left[1+A_1(\rho^{\frac\alpha2-\frac12}+\eps_1+\rho^\frac{\alpha}{2}\eps_0)\right]|x-y|^\alpha.    
\end{align*}
To estimate $\left|\frac{\pi(y,u)}{\pi(x,u)}-1\right|$, let us assume without loss of generality that $\pi(y,u)\geq \pi(x,u)$. Then since $|g(x,u)-g(y,u)|\leq [g(\cdot, u)]_{0,\alpha}|x-y|^\alpha$ and $|U|=1$, we find
\begin{align*}
\left|\frac{\pi(y,u)}{\pi(x,u)}-1\right|&= \frac{\exp[g(y,u)]\int_U \exp[g(x,u')]\,du'}{\int_U\exp [g(y,u')]\,du' \exp[g(x,u)]}-1\\
&\leq \frac{\exp[g(y,u)]\int_U \exp\left[g(y,u')+[g(\cdot, u)]_{0,\alpha}|x-y|^\alpha\right]\,du'}{\int_U\exp [g(y,u')]\,du' \exp[g(x,u)]}-1 \\
&= {\exp[g(y,u)-g(x,u)]\exp\left[\sup_{u\in U}[g(\cdot, u)]_{0,\alpha}|x-y|^\alpha\right]}-1 \\
&\leq \exp\left[C\left(1+A_1(\rho^{\frac\alpha2-\frac12}+\eps_1+\rho^\frac{\alpha}{2}\eps_0) \right)\right]|x-y|^\alpha,    
\end{align*}
for all $|x-y|\leq 1$.
Finally, combining these with \eqref{1.3}, we obtain
\[
\left|\int_U\left(\pi(x,u)\ln \pi(x,u)-\pi(y,u)\ln \pi(y,u)\right)\,du\right|\leq \exp\left[C\left(1+A_1(\rho^{\frac\alpha2-\frac12}+\eps_1+\rho^\frac{\alpha}{2}\eps_0) \right)\right]|x-y|^\alpha.
\]
\end{proof}

Recall that, by \eqref{eq:PI}, $v^n$ satisfies the linear parabolic equation:
\beq\lb{1.0}
	\rho v^n(x)- {r}^n(x) - {b}^n(x) \cdot {D}  v^n(x)-2^{-1}\tr({\Sigma}^n(x) {D} ^2 v^n(x)) +\lambda {\Hc}^n(x)=0,
\eeq
where $\pi^n$ is given in \eqref{1.11}, and
\beq\lb{1.10}
\begin{aligned}
&{r}^n(x):=\int_U r(x,u)\pi^n(x,u)\,du,\quad &&{b}^n(x):=\int_U b(x,u)\pi^n(x,u)\,du,
\\
&{\Sigma}^n(x):=\int_U \sigma\sigma^T(x,u)\pi^n(x,u)\,du,\quad &&{\Hc}^n(x):=\int_U \ln(\pi^n(x,u))\pi^n(x,u)\,du.    
\end{aligned}
\eeq

\medskip

Assume that conditions \eqref{cond1}, \eqref{cond2} hold and that $|U|=1$, it follows from the proof of Lemma \ref{L.2} that the SDE corresponding to iteration step $n$, given by
$$
	dX^{\pi^n}_t = b^n(X^{\pi^n}_t)dt+\sqrt{\Sigma^n(X^{\pi^n}_t)}\,dW_t
$$
	admits a unique strong solution, provided $\|v^{n-1}\|_{2,\alpha}$ is finite. Then by selecting $\|v^0\|_{2,\alpha}<\infty$ for some $\alpha\in(0,1)$, an induction argument ensures that the above SDE admits a unique strong solution for all $n=1,2,3,\dots$.
	Further, the value function $V^{\pi^n}$ equals to the unique solution $v^n$ of \eqref{eq:PI}. In addition, by taking $\rho$ sufficiently large, the sequence $v^n$ satisfies the following uniform $\Cc^{2,\alpha}$ estimate.


\begin{Theorem}\lb{T.1.1}
Let $\alpha\in (0,1)$, and assume \eqref{cond1}, \eqref{cond2}, and $|U|=1$. Then there exists  $A_1\geq 1$ depending only on $\|v^0\|_{2,\alpha}$, $C_0$, $\lambda$ and $\eta$ from Lemma \ref{L.1} such that if
\beq\lb{cond3}
\rho\geq A_1^{2/(1-\alpha)},\quad \eps_0 A_1\rho^\frac{\alpha}{2}\leq 1,\quad\text{and}\quad \eps_1 A_1\leq 1,
\eeq
we have for $v^n$ from PIA,
\[
\|v^n\|_{2}\leq A_1\quad\text{and}\quad \rho^{-\frac\alpha2}[v^n]_{2,\alpha}\leq A_1\quad \text{for all $n\geq 1$}.
\]
\end{Theorem}

We note that both \(\eps_0\) and \(\eps_1\) are quantifiable from the proof using Schauder estimates, and \(\eps_1\) can be chosen independently of \(\rho\). In other words, \(M(x, u) = \Sigma(x, u) - \Sigma_0(x)\) needs to be small for large \(\rho\), although it is allowed to have certain oscillations independent of \(\rho\).

\begin{proof}
Let $A\geq 1$ to be determined  depending only on  $\lambda$ and \eqref{cond1}, and let $A_1= \eta(A)>0$. Since $v^0\in\Cc^{2,\alpha}(\R^d)$, then we can have that the H\"{o}lder norms of $r^1, b^1,\Sigma^1,\Hc^1$ are bounded by $A$, i.e. \eqref{2.1'} holds with $r^1-\lambda\Hc^1,b^1,\Sigma^1$ in place of $\tilde{r},\tilde{b}, \tilde{\Sigma}$. Thus Lemma \ref{L.1} implies that
\[
\rho\|v^{1}\|_{\infty},\; \rho^{1-\frac\alpha2}[v^{1}]_{0,\alpha} ,\; \rho^{\frac12-\frac\alpha2}[v^{1}]_{1,\alpha},\; \rho^{-\frac\alpha2}[v^{1}]_{2,\alpha},\; \rho^\frac12\llbracket v^{1}\rrbracket_1,\; \llbracket v^{1}\rrbracket_2\leq A_1.
\]
Let us assume for induction that for some $n\geq 2$, the following holds
\[
\rho\|v^{n-1}\|_{\infty},\; \rho^{1-\frac\alpha2}[v^{n-1}]_{0,\alpha} ,\; \rho^{\frac12-\frac\alpha2}[v^{n-1}]_{1,\alpha},\; \rho^{-\frac\alpha2}[v^{n-1}]_{2,\alpha},\; \rho^\frac12\llbracket v^{n-1}\rrbracket_1,\; \llbracket v^{n-1}\rrbracket_2\leq A_1.
\]
We now prove all the above inequalities hold with $v^n$ in place of $v^{n-1}$.

It follows from Lemma \ref{L.2} and $|U|=1$ that
\begin{align*}
\|r^n(\cdot)\|_{0,\alpha} &\leq \|r(\cdot)\|_{0,\alpha}+\|r(\cdot)\|_{\infty} \sup_{u\in U}[\pi^n(\cdot,u)]_{0,\alpha}\\
&\leq \left(1+A_1(\rho^{\frac\alpha2-\frac12}+\eps_1+\rho^\frac{\alpha}{2}\eps_0)\right)\exp\left[C(1+A_1(\rho^{-\frac12}+\eps_0 )\right],   
\end{align*}
where $C$ only depending on $\lambda$ and $C_0$ in \eqref{cond1}.
Since $A_1=\eta(A)$, \eqref{cond3} and the above yield
\[
\|r^n(\cdot)\|_{0,\alpha}\leq 4\exp(3C) \leq A/2
\]
if $A$ 
is sufficiently large compared to $C$.
For the entropy, Lemma \ref{L.2} and \eqref{cond3} yield
\[
\left\|\lambda\int_U\pi^n(\cdot,u)\ln \pi^n(\cdot,u)\,du\right\|_{0,\alpha}\leq \exp\left[C\left(1+A_1(\rho^{\frac\alpha2-\frac12}+\eps_1+\rho^\frac{\alpha}{2}\eps_0)\right)\right]\leq \exp(4C)\leq A/2
\]
if $A$ is sufficiently large depending only on $C$. Thus, $\|r^n-\lambda\Hc^n\|_{0,\alpha}\leq A$. Similarly, can we get
\[
\| b^n(\cdot)\|_{0,\alpha} ,\quad \|\Sigma^n(\cdot)\|_{0,\alpha}\leq A.
\]

Finally, we apply Lemma \ref{L.1} to conclude that 
\[
\rho\|v^{n}\|_{\infty},\quad  \rho^{1-\frac\alpha2}[v^{n}]_{0,\alpha},\quad \rho^{\frac12-\frac\alpha2}[v^{n}]_{1,\alpha},\quad \rho^{-\frac\alpha2}[v^{n}]_{2,\alpha},\quad \rho^\frac12\llbracket v^{n}\rrbracket_1,\quad \llbracket v^{n}\rrbracket_2\leq \eta(A)=A_1.
\]
Since the constants $C,A,A_1$ are independent of $n$, the proof is finished by induction.
\end{proof}

\section{Convergence for PIA with bounded coefficients}\label{sec: converge.bounded}



\subsection{Convergence for uniform $\Cc^{2,\alpha}$ solutions}

Theorem \ref{T.1.1} has proved that, under certain conditions, $v^n$ from PIA are uniformly $\Cc^{2,\alpha}$. The following theorem shows that such solutions converge to the solution $v^*$ of \eqref{eq:HJB}.

\begin{Theorem}\label{thm:bound.converge}
    Assume \eqref{cond1}. 
    Let $v^*$ solve \eqref{eq:HJB} and let $v^n$ be the solution to \eqref{eq:PI} for $n\in \N$.
    If $v^n$ are uniformly bounded in $\Cc^{2,\alpha}(\R^d)$ for all $n\in \N$,
    then $\lim_{n\to \infty}v^n=v^*$ as $n\to\infty$ locally uniformly in $\Cc^{2,\alpha}(\R^d)$.
\end{Theorem}

\begin{proof}
Since $\pi^{n+1}$ is a maximizer for $F_\pi(x,Dv^n,D^2v^n)$ for each $n\geq 0$, it follows from \eqref{0.2} and the equation \eqref{eq:PI} that
\[
\rho v^n-F(x,Dv^n,D^2v^n)=\rho v^n-F_{\pi^{n+1}}(x,Dv^n,D^2v^n)\leq 0.
\]
Thus, the comparison principle yields 
\[
v^n\leq v^{n+1}\leq v^*.
\]
We can then take $\bar{v}:=\lim_{n\to \infty}v^n$. Since $v^n$ are uniformly bounded in $\Cc^{2,\alpha}(\R^d)$ for all $n$, we get that $v^n\to\bar v$ as $n\to\infty$ locally uniformly in $\Cc^{2,\alpha}(\R^d)$, and $\bar v\in \Cc^{2,\alpha}(\R^d)$. 

Note that, for each $u\in U$,
\[
\pi^n(x,u)=\Gamma(x,Dv^n,D^2v^n)(u)\to \Gamma(x,D\bar v,D^2\bar v)(u)\quad\text{ locally uniformly}.
\]
By the equations of $v^n$ and the stability of viscosity solutions (under locally uniform convergence), we get that $\bar{v}$ is a viscosity 
to
\begin{align*}
\rho v- \int_U \big[r(x,u)+& b(x,u) \cdot Dv+ 2^{-1}\tr(\sigma\sigma^T(x) D^2 v)\\
&-\lambda \ln(\Gamma(x,Dv,D^2v)(u))\big] \Gamma(x,Dv,D^2v)(u)\,du =0.  
\end{align*}
In view of the definition of $\Gamma$ in \eqref{0.1}, this shows that $\bar{v}$ is a solution to \eqref{eq:HJB}. By the comparison principle again, we have $\bar{v}=v^*$, which finishes the proof.
\end{proof}

\subsection{Quantitative convergence results}
In this subsection, we estimate the convergence rate of $v^n$ to $v^*$, allowing some errors. 
Let $\pi^*(x,u)=\Gamma(x, D v^*, D^2v^*)(u)$ be the unique optimal feedback control, and let
\beq\lb{3.6}
\begin{aligned}
&{r}^*(x)=\int_U r(x,u)\pi^*(x,u)\,du,\quad &&{b}^*(x)=\int_U b(x,u)\pi^*(x,u)\,du,\\
&{\Sigma}^*(x)=\int_U \sigma\sigma^T(x,u)\pi^*(x,u)\,du,\quad &&{\Hc}^*(x)=\int_U \ln(\pi^*(x,u))\pi^*(x,u)\,du,    
\end{aligned}
\eeq
and then \eqref{eq:HJB} is reduced to
 \beq\lb{2.1}
 \rho v^*(x)-r^*(x)-b^*(x)\cdot  Dv^*(x)-2^{-1}\tr(\Sigma^*(x) D^2 v^*(x))+\lambda\Hc^*(x)=0.
 \eeq


\begin{Theorem}\lb{T.2.1}
Assume \eqref{cond1}, \eqref{cond2} for some $C_0\geq 1$, and $\alpha,\eps_0,\eps_1\in (0,1)$. Assume $|U|=1$, and that 
$\Sigma_0(\cdot)$ is Lipschitz continuous. Then there exists $A_1\geq 1$ such that if \eqref{cond3} holds, we have for all $n\geq 1$ and $x_0\in \R^n$,
\[
\rho\int_{B_1(x_0)} | v^n-v^*|^2 \,dx+\int_{B_1(x_0)} |D (v^n-v^*)|^2\, dx \leq C2^{-n}+C'\eps_0^2/\rho,
\]
where $C>0$ depends only on $d$ and $\|v^0\|_{2}+\|v^*\|_{2}$, and $A_1,C'>0$ also depend on $\lambda,C_0$, and $\eta$.
As a consequence, we have for all $n\geq 1$,
\[
\|D(v^n-v^*)\|_\infty\leq (C2^{-n}+C'\eps_0^2/\rho)^{\frac{1}{d+2}}
\]
and
\[
\|v^n-v^*\|_\infty\leq \rho^{-\frac{1}{d+2}}(C2^{-n}+C'
\eps_0^2/\rho)^{\frac{1}{d+2}+\frac{d}{(d+2)^2}}.
\]   
\end{Theorem}

\begin{Remark}\label{rem:quant}
If $\sigma$ is independent of $u$, then $\eps_0=\eps_1=0$, so we obtain the exponential convergence rate for PIA as a corollary of the result. Also, we remark that the smallness of $\eps_1$ is only used to guarantee that $v^n\in \Cc^{2,\alpha}(\R^d)$ uniformly for all $n\geq 1$.
\end{Remark}

\begin{proof}
Without loss of generality, let us assume that $x_0=0$. The proof will be divided into three steps.

{\bf Step 1.} 
Recall that, under the assumptions, $v^n$ and $v^*$ are uniformly bounded in $\Cc^{2,\alpha}(\R^d)$ for all $n\geq1$. Denote $M(x,u)=\Sigma(x,u)-\Sigma_0(x)$ and then we have
\[
|M(x,u)|\leq \eps_0 \qquad\text{for all }x,u.
\]
As done before, if we write
\[
g^n(x,u):=\frac{1}{\lambda}\left(r(x,u)+ b(x,u) \cdot Dv^n+ \frac{1}{2}\tr(M(x,u)D^2v^n)\right),
\]
\[
g^*(x,u):=\frac{1}{\lambda}\left(r(x,u)+ b(x,u) \cdot Dv^*+ \frac{1}{2}\tr(M(x,u)D^2v^*)\right),
\]
then
\[
\pi^n(x,u)=\frac{\exp(g^{n-1}(x,u))}{\int_U\exp(g^{n-1}(x,u'))\,du'},\qquad \pi^*(x,u)=\frac{\exp(g^*(x,u))}{\int_U\exp(g^*(x,u'))\,du'}.
\]
In view of \eqref{1.0} and \eqref{2.1}, and setting $\bar v^n:=v^*-v^n$, we obtain
\beq\lb{2.5}
\rho \bar v^n(x)-\frac{1}{2}\tr(\Sigma_0(x) D^2 \bar v^n)= \lambda\int_U \left( g^*(x,u)\pi^*(x,u)-g^n(x,u)\pi^n(x,u) \right)du-\lambda(\Hc^*(x)-\Hc^n(x)). 
\eeq

It follows from the proof of \eqref{1.2} and the assumption of \eqref{cond3} that $|g^n|,|g^*|\leq C$ uniformly for all $n\geq 1$. 
By \eqref{1.11} and the uniform regularity of $v^n,v^*$, we get
\[
|\pi^n(x,u)- \pi^*(x,u)|\leq C\left(  |D \bar v^{n-1}| +\eps_0|D^2 \bar v^{n-1}| \right).
\]
Let us remark that $\eps_1$ is not needed here.
For simplicity of notation, we denote 
\[
\Ec_{n-1}: =\Ec_{n-1} (x)=C\left(  |D \bar v^{n-1}| +\eps_0|D^2 \bar v^{n-1}| \right)(x)
\]
with possibly different constant $C$ in $\Ec_{n}$ from one line to another. Recall \eqref{1.10} and \eqref{1.11}.
By the assumptions, it follows that
\[  
|b^*-b^n|,\, |r^*-r^n|,\, |M^*-M^n|\leq \Ec_{n-1} .
\]
Thus, there are $C$ and $\Ec_{n-1} $ such that for all $u\in U$, $\left|\pi^*(x,u)-\pi^n(x,u)  \right|\leq \Ec_{n-1} $. Then
\begin{align*}
&\left| g^*(x,u)\pi^*(x,u)-g^n(x,u)\pi^n(x,u)  \right|
\\
&\qquad\leq |g^*(x,u)|\left| \pi^*(x,u)-\pi^n(x,u)  \right|+|\pi^n(x,u)|\left| g^*(x,u)-g^n(x,u)  \right|\leq \Ec_{n-1} +\Ec_{n} .    
\end{align*}

To estimate $|\Hc^n-\Hc^*|$, note that
\begin{align*}
\left|\ln \pi^n(x,u)-\ln\pi^*(x,u)\right|=\left|g^{n-1}(x,u)-g^*(x,u)-\ln \int_U e^{g^{n-1}(x,u')-g^*(x,u')} \,du'\right|\leq \Ec_{n-1} , 
\end{align*}
and, since $\pi^n$ is strictly positive by the uniform regularity of $v^n$,
\begin{align*}
\left|\frac{\pi^*(x,u)}{\pi^n(x,u)}-1\right|&\leq C|{\pi^*(x,u)}-{\pi^n(x,u)}|\leq \Ec_{n-1} .
\end{align*}
Thus, we obtain
\begin{align*}
|\Hc^n(x)-\Hc^*(x)|&\leq \left|\int_U \pi^n(x,u)\ln \pi^n(x,u)\,du-\int_U \pi^*(x,u)\ln \pi^*(x,u)\,du \right|\\
&\leq \sup_{u\in U}|\ln \pi^n(x,u)-\ln\pi^*(x,u)|+\sup_{u\in U} \left|\frac{\pi^*(x,u)}{\pi^n(x,u)}-1\right||\ln \pi^*(x,u)|\leq\Ec_{n-1} .    
\end{align*}
Putting these into \eqref{2.5} yields 
\beq\lb{2.5'}
\rho \bar v^n(x)-2^{-1}\tr(\Sigma_0(x) D^2 \bar v^n)\leq \Ec_{n-1} +\Ec_{n} 
\eeq
for some $C$ depending only on the assumptions.

\medskip

{\bf Step 2.} 
Next, let $\phi\in [0,1]$ be a smooth function on $\R^n$ such that for some $C>0$,
\[
\phi(\cdot)\equiv 1\text{  on  }B_1,\quad |D\phi(x)|\leq C\phi(x)\text{ for all }x\in\R^d,\quad \int_{\R^d}\phi(x) \,dx<\infty.
\]
Such $\phi$ exists as one can take a smooth version of $\min\{1, e^{2-|x|}\}$. 

From \eqref{2.5'}, multiply $\bar{v}^n\phi$ on both sides of  \eqref{2.5} and integrate over $\R^n$ to get
\beq\lb{2.9}
\rho\int_{\R^d} |\bar v^n|^2\phi\, dx-\frac{1}{2}\int_{\R^d} \tr(\Sigma_0 D^2 \bar v^n) \bar v^n\phi\, dx\leq \int_{\R^d} (\Ec_{n-1}+\Ec_{n})\bar v^n\phi\,dx.
\eeq
Since $\Sigma_0=\Sigma_0(x)$ is uniformly Lipschitz continuous and uniformly elliptic, direct computation yields
\begin{align*}
-\int_{\R^d} &\tr(\Sigma_0 D^2 \bar v^n) \bar v^n\phi\, dx =\sum_{1\leq i,j\leq d}\int_{\R^d} \partial_{x_j}(\Sigma_0\bar v^n\phi)_{ij}\partial_{x_i}\bar{v}^n \,dx\\
&\quad =\sum_{1\leq i,j\leq d}\int_{\R^d}\partial_{x_j}\left((\Sigma_0)_{ij}\phi\right)\bar{v}^n\partial_{x_i}\bar{v}^n \,dx+\sum_{1\leq i,j\leq d}\int_{\R^d}\partial_{x_j}\bar v^n (\Sigma_0)_{ij}\phi\partial_{x_i}\bar{v}^n \,dx\\
&\quad\geq c \int_{\R^d}|D \bar v^n|^2\phi\,dx -C \int_{\R^d}|\bar v^n||D \bar v^n|\phi\,dx\geq c\int_{\R^d} |D \bar v^n|^2\phi\,dx -C\int_{\R^d} |\bar v^n|^2\phi\,dx.
\end{align*}
In the first inequality, we used that $|D\phi|\leq C\phi$ and uniform ellipticity; and we applied Young's inequality in the second inequality. The positive constants $c,C$ depend only on $\phi$ and the $\Cc^1$ norm of $\Sigma_0$.
With this, \eqref{2.9} implies that for some positive constants $C_1,c_1,C_2>0$,
\[
\rho\int_{\R^d} |\bar v^n|^2\phi\, dx +c_1\int_{\R^d} |D\bar v^n|^2\phi\, dx
\leq C_2\int_{\R^d} |\bar v^n|^2\phi\,dx+C_1\int_{\R^d}\left( \Ec_{n-1}+\Ec_{n}\right)|\bar v^n|\phi \,dx.
\]
By Young's inequality and that $|D^2 \bar v^{n-1}|,|D^2 \bar v^{n}|\leq C$, we obtain
\[
\begin{aligned}
&(\rho-C_2)\int_{\R^d} |\bar v^n|^2\phi\, dx +c_1\int_{\R^d} |D\bar v^n|^2\phi\, dx\\
&\quad\leq C_1\int_{\R^d} (|D\bar v^{n-1}|+ |D\bar v^{n}|)|\bar v^n|\phi\,dx+CC_1\eps_0\int_{\R^d}|\bar v^n|\phi \,dx\\
&\quad\leq  \frac{c_1}{2}\int_{\R^d} |D\bar v^{n}|^2\phi\, dx+\frac{c_1}{4}\int_{\R^d} |D\bar v^{n-1}|^2\phi\, dx+\frac{3C_1^2}{2c_1}\int_{\R^d} |\bar v^n|^2\phi \, dx+C_1'\eps_0 \left(\int_{\R^d} |\bar{v}^n|^2\phi\,dx\right)^{1/2} 
\end{aligned}
\]
with $C_1'=CC_1$.
Let us assume $\rho\geq 2C_2+{3C_1^2}/{c_1}$, and so
\beq\lb{2.10}
\begin{aligned}
{\rho}\int_{\R^d} |\bar v^n|^2\phi\, dx +c_1\int_{\R^d} |D\bar v^n|^2\phi\, dx
\leq  \frac{c_1}{2}\int_{\R^d} |D\bar v^{n-1}|^2\phi\, dx+2C_1'\eps_0 \left(\int_{\R^d} |\bar{v}^n|^2\phi\,dx\right)^{1/2}.    
\end{aligned}
\eeq

If $\int_{\R^d} |\bar v^n|^2\phi\, dx \geq (\frac{2C_1'\eps_0}{\rho})^2$, \eqref{2.10} is reduced to
\[
\int_{\R^d} |D\bar v^n|^2\phi\, dx
\leq  \frac{1}{2}\int_{\R^d} |D\bar v^{n-1}|^2\phi\, dx.    
\]
Otherwise, $\int_{\R^d} |\bar v^n|^2\phi\, dx \leq (\frac{2C_1'\eps_0}{\rho})^2$, so for both cases we have
\[
\int_{\R^d} |D\bar v^n|^2\phi\, dx
\leq  \frac{1}{2}\int_{\R^d} |D\bar v^{n-1}|^2\phi\, dx+\frac{(2C_1'\eps_0)^2}{\rho c_1}.
\]
It follows that
\[
\begin{aligned}
\int_{\R^d} |D\bar v^{n}|^2\phi\, dx
&\leq  \frac{1}{2}\int_{\R^d} |D\bar v^{n-1}|^2\phi\, dx+\frac{(2C_1'\eps_0)^2}{\rho c_1}\leq  \frac{1}{4}\int_{\R^d} |D\bar v^{n-2}|^2\phi\, dx+(1+\frac12)\frac{(2C_1'\eps_0)^2}{\rho c_1}\\
& \leq \ldots \leq 2^{-n}\int_{\R^d} |D\bar v^{0}|^2\phi\, dx+\frac{8(C_1'\eps_0)^2}{\rho c_1}.
\end{aligned}
\]
Since $\|v^0\|_{2},\|v^*\|_{2}<\infty$ and $\int_{\R^d} |\phi(x)|dx<\infty$, we proved that there exist $C_3$ depending only on $d$ and $\|v^0\|_{2}+\|v^*\|_{2}$, and $C_4$ depending on $\lambda,C_0$, and $\eta$ such that for all $n\geq 1$,
\beq\lb{2.11}
\int_{\R^d} |D\bar v^{n}|^2\phi\, dx\leq C_32^{-n}+C_4\eps_0^2/\rho.
\eeq

{\bf Step 3.} Finally, it follows from \eqref{2.10} that
\[
{\rho}\int_{\R^d} |\bar v^n|^2\phi\, dx 
\leq  c_1C_32^{-n}+c_1C_4\eps_0^2/\rho +2C_1'\eps_0 \left(\int_{\R^d} |\bar{v}^n|^2\phi\,dx\right)^{1/2}.
\]
This shows that
\[
\rho\int_{\R^d} |\bar v^n|^2\phi\, dx\leq \max\left\{2c_1\left(C_32^{-n}+C_4{\eps_0^2}/\rho\right),\, (4C_1'\eps_0)^2/\rho \right\},
\]
which, combining with \eqref{2.11} and the fact that $\phi=1$ in $B_1$, finishes the proof of the first claim.

Because $\sup_{n}\|v^n-v^*\|_{2,\alpha}<\infty$, if $|D(v^n-v^*)|(0)=\delta>0$ for some $\delta>0$, then $|D(v^n-v^*)|\geq\frac\delta2$ in $B_{c\delta}$ for some $c>0$. Therefore $\|D(v^n-v^*)\|^2_{L^2(B_1)}\leq (C2^{-n}+C'\eps_0^2/\rho)$ from the first part of the theorem yields
\[
|D(v^n-v^*)|(0)\leq (C2^{-n}+C'\eps_0^2/\rho)^{\frac{1}{d+2}}=:\eps_2.
\]
With this, if $|v^n-v^*|(0)=\delta'>0$ for some $\delta'>0$, then $|v^n-v^*|\geq\frac{\delta'}{2}$ in $B_{c\delta'/\eps_2}\cap B_1$ for some $c>0$. We obtain the following pointwise estimate:
\[
|v^n(0)-v^*(0)|\leq \rho^{-\frac{1}{d+2}}(C2^{-n}+C'\eps_0^2/\rho)^{\frac{1}{d+2}+\frac{d}{(d+2)^2}}.
\]   
After shifting the solutions, these finish the proof.

\end{proof}

\section{Unbounded degenerate elliptic equations}\label{sec:unbound.nosigma}

In this section, we study degenerate elliptic equations, in which the coefficients might be unbounded in $\bbR^d$. The well-posedness results established here are of independent interest and will be applied later in Section \ref{subsec:unbound.convergence} to analyze the convergence of PIA in the second scenario.

Recall that, by \eqref{0.3},
$v^*(x)=\sup_{\pi\in \Pc(U)}V^\pi(x)$ is a viscosity solution to
\[
\rho v-F(x,Dv,D^2v)=0=\rho v-F_{\pi^*}(x,Dv,D^2v)\quad\text{ with }\pi^*=\Gamma(x,Dv^*,D^2v^*).
\] 
Plugging in the definition of $\Gamma$ in \eqref{0.1} into the definition of $F_\pi$ in \eqref{eq:Fpi} yields the equation
\begin{equation}
\label{4.0}
\begin{aligned}
\rho v-\lambda \ln \int_{U} \exp\bigg[\frac{1}{\lambda} \bigg(r(x,u) + b(x,u) \cdot D v
+ \frac12\tr(\Sigma(x,u) D^2v)\bigg) \bigg] d u= 0.
\end{aligned}
\end{equation}
In this section, we study \eqref{4.0} and only assume $\Sigma\geq 0$. 
If $r,b$, and $\Sigma$ are independent of $u$, the equation becomes a linear equation, which includes \eqref{eq:PI}.


Recall $|{U}|=1$ and $\Sigma(x,u)=\sigma(x,u)\sigma(x,u)^T$.
We make the following assumptions:
\beq\lb{c.1}
\text{$r(x,u),\, b(x,u),\, \sigma(x,u)$ are locally uniformly Lipschitz continuous in $x$, }
\eeq
and
there exist $N>0$ and $A_0,A_1\geq 1$ such that for all $(x,u)\in \R^d\times{U}$,
\beq\lb{c.3}
 |r(x, u)|\leq A_1 (1+|x|^{N}),\,\,|b(x,u)|\leq A_2(1+|x|),\,\,|\sigma(x,u)|\leq A_3(1+|x|),
\eeq
where, for $X=(x_{ij})$ a vector or a matrix, $|X|:=\max_{i,j}|x_{ij}|$.



\subsection{Existence and uniqueness}\label{subsec:unbound.unique}

We start with a comparison principle in bounded domains. Throughout this section, we allow degenerate diffusion, that is, we only require $\Sigma \geq 0$.

\begin{Lemma}\lb{T.1.2}
Let $\Omega\subseteq\bbR^d$ be a bounded open set, and assume \eqref{c.1}.
Let $\mu$ (resp. $v$) be a bounded subsolution (resp. supersolution) to 
\beq\lb{4.1}
\rho \mu-F(x,D\mu,D^2\mu)=0 \quad\text{ in $\Omega$,}
\eeq
such that
\[
\sup_{x\in\partial\Omega}(\mu(x)-v(x))\leq 0. 
\]
Then $\mu\leq v$ in $\Omega$. 
\end{Lemma}

\begin{proof}
In view of \cite[Theorem 3.3]{user}, it suffices to check that the operator $\rho \mu-F(x,p,X)$ is proper and satisfies conditions (3.13) and (3.14) in \cite{user}. By the definition of the operator, we only  need to verify (3.14):
there is a function $\omega:[0,\infty]\to[0,\infty]$ such that $\omega(0+)=0$ and
\beq\lb{314}
F(x,\alpha(x-y),X_\alpha)-F(y,\alpha(x-y),Y_\alpha)\leq \omega(\alpha|x-y|^2+|x-y|)
\eeq
whenever $x,y\in\Omega$, $r\in\bbR$, and $X_\alpha,Y_\alpha\in\Sc^d$ are such that
\beq\lb{4.6}
-3\alpha\begin{pmatrix}
\I_d & 0\\
0 & \I_d
\end{pmatrix}\leq
\begin{pmatrix}
X_\alpha & 0\\
0 & -Y_\alpha
\end{pmatrix}
\leq 3\alpha\begin{pmatrix}
\I_d & -\I_d\\
-\I_d & \I_d
\end{pmatrix}.
\eeq

Indeed from \cite[Example 3.6]{user}, \eqref{4.6} and the uniform Lipschitz continuity of $\sigma(\cdot,u)$ yield for any $u\in{U}$ and $x,y\in\Omega$,
\[
\tr(\sigma(x,u) \sigma^T(x,u)X_\alpha)-\tr(\sigma(y,u) \sigma^T(y,u)Y_\alpha)\leq C\alpha|x-y|^2
\]
where $C>0$ only depends on \eqref{c.1}, $A_1$, and $R$, where $R$ is such that $\Omega\subseteq B_R$. Moreover, using \eqref{c.1} yields, for
\beq\lb{4.7}
f_u(x,p,X):=r(x,u) + b(x,u) \cdot p
+ \frac12\tr(\Sigma(x,u)X),
\eeq
that
\[
f_u(x,p,X_\alpha)-f_u(y,p,Y_\alpha)\leq C\left(|x-y|+|x-y||p|+\alpha{|x-y|^2}\right)
\]
for some $C>0$, and so
\begin{align*}
&\lambda\ln\int_{U} \exp\left[\frac{1}{\lambda}f_u(x,\alpha(x-y),X_\alpha)\right]d u - \lambda\ln\int_{U} \exp\left[\frac{1}{\lambda}f_u(y,\alpha(x-y),Y_\alpha)\right]d u \\
&\qquad\qquad \leq \lambda\ln\int_{U} \exp\left(\frac{1}{\lambda}C(|x-y|+\alpha|x-y|^2)\right)d u =C(|x-y|+\alpha|x-y|^2).
\end{align*}
This and the equation \eqref{4.0} show \eqref{314} with $\omega(z):=Cz$.
\end{proof}

\begin{Lemma}[Comparison principle in $\bbR^d$] \lb{L.cp}
Assume \eqref{c.1}, \eqref{c.3}, and
\beq\lb{c.4}
\rho\geq 4(N+1)(A_2+NA_3).
\eeq
Let  $\mu$ and $v$ be, respectively, a  subsolution and a  supersolution to  \eqref{4.0} in
$\bbR^d$ such that
\beq\lb{4.9}
\limsup_{|x|\to\infty}\frac{\mu(x)-v(x)}{|x|^{N+1}}\leq 0.
\eeq
Then $\mu\leq v$ in $\bbR^d$. 
\end{Lemma}

\begin{proof}
For any $\eps>0$, define
\[
\mu_\eps(x):=\mu(x)-\eps (1+|x|^{N+1}).
\]
We claim that $\mu_\eps$ is a (viscosity) subsolution to \eqref{4.0} in $\bbR^d$.
Indeed, if $\varphi\in \mathcal{C}^\infty(\bbR^d)$ is such that $\mu_\eps-\varphi$ has a local maximum at $x_0\in\bbR^d$. Then $\mu-\varphi_\eps$ with $\varphi_\eps:=\varphi+\eps (1+|x|^{N+1})$ has a local maximum at $x_0$.
Using the notation from \eqref{4.7} and the assumptions \eqref{c.1}--\eqref{c.3} implies that 
\begin{align*}
f_u(x_0,D\varphi_\eps(x_0),&\, D^2\varphi_\eps(x_0))-f_u(x_0,D\varphi(x_0),D^2\varphi(x_0))\\
&\leq \eps(N+1)|b(x_0,u)||x_0|^N+\eps(N+1)N|\Sigma(x_0,u)||x_0|^{N-1}\\
&\leq \eps(N+1)A_2(1+|x_0|)|x_0|^N+\eps(N+1)NA_3(1+|x_0|)^2|x_0|^{N-1}.
\end{align*}
Recall that $F(x,p,X)=\lambda\ln \int_U\exp[\frac1\lambda f_u(x,p,X)]\,du$.
Since $\mu$ is a subsolution to \eqref{4.0},  
\[
\rho \mu(x_0)-F(x_0,D\varphi_\eps(x_0),D^2\varphi_\eps(x_0))\leq 0.
\]
We obtain at $x=x_0$,
\begin{align*}
\rho \mu_\eps-F(x_0,&D\varphi,D^2\varphi)\leq \rho (\mu-\eps(1+|x_0|^{N+1}))-F(x_0,D\varphi_\eps,D^2\varphi_\eps)\\
&\quad+\eps(N+1)A_2(1+|x_0|)|x_0|^N+\eps(N+1)NA_3(1+|x_0|)^2|x_0|^{N-1}\\
&\leq -\eps\rho\left( 1+|x_0|^{N+1}\right)+\eps(N+1)(A_2+NA_3)(1+|x_0|)^2|x_0|^{N-1}.
\end{align*}
Hence, by \eqref{c.4}, we get from the above that
\[
\rho \mu_\eps-F(x_0,D\varphi,D^2\varphi)\leq 0.
\]
Therefore, for all $\eps>0$, $\mu_\eps$ is a subsolution to \eqref{4.0}.

Now by \eqref{4.9}, there exists $R_\eps>0$ such that $\lim_{\eps\to0}R_\eps=\infty$ and $\mu_\eps(x)\leq v(x)$ for all $|x|\geq R_\eps$. Therefore, applying Lemma \ref{T.1.2} to $v,\mu_\eps$ with $\Omega=B_{R_\eps}$ yields
\[
\mu_\eps(x)\leq v(x)\quad\text{ for all }x\in B_{R_\eps}.
\]
Taking $\eps\to 0$ leads to $\mu\leq v$ in $\bbR^d$.
\end{proof}

\begin{Proposition}\lb{P.4.1}
Under the assumption of Lemma \ref{L.cp},
there exists a unique (viscosity) solution $v$ to \eqref{4.0} such that for all $x\in\R^d$,
\beq\lb{4.8}
|v(x)|\leq 2A_1\rho^{-1} (1+|x|^2)^{N/2}.
\eeq
\end{Proposition}
\begin{proof}
To prove the existence and uniqueness of solutions, in view of the comparison principle, it suffices to produce a supersolution and a subsolution with a polynomial growth at infinity and invoke Perron's method.

By \eqref{4.0} and \eqref{c.3}, for any $(x,p,X)\in\bbR^d\times\bbR^{d}\times \Sc^d$,  we have 
\beq\lb{2.6}
 F(x,p,X)\leq A_1(1+|x|^2)^{N/2}+A_2(1+|x|^2)^{1/2}|p|+A_3(1+|x|^2)|X|.
\eeq
Set $\phi(x):=(1+|x|^2)^{N/2}$, and then define $\bar{\mu}(x):=A_0\phi(x)$ with $A_0:=2A_1/\rho$.
For simplicity, below we drop $(x)$ from the notations of $\bar{\mu}(x)$, $\phi(x)$. 
We have from \eqref{2.6} and direct computations that
\begin{align*}
\rho \bar{\mu}- F(x,D \bar{\mu},D^2\bar{\mu})&\geq \rho A_0(1+|x|^2)^{N/2}-A_1(1+|x|^2)^{N/2}-A_0N(A_2+A_3(N-1))(1+|x|^2)^{N/2}\\
&\geq (\rho A_0/2-A_1)(1+|x|^2)^{N/2}\geq 0,    
\end{align*}
thanks to \eqref{c.4}. 
Thus, $\bar{\mu}$ is a supersolution. 

Similarly, it can be shown that $\underline{\mu}:=-\bar{\mu}$ is a subsolution. It is clear that 
\[
\lim_{|x|\to\infty}\frac{|\bar{\mu}(x)|+|\underline{\mu}(x)|}{|x|^{N+1}}=0.
\]
Thus by Perron's method and Lemma \ref{L.cp}, we obtain the unique solution $v$ to \eqref{4.0} such that $\underline{\mu}\leq v\leq \bar{\mu}$, which yields \eqref{4.8}. 
\end{proof}

\smallskip

\begin{Remark}\label{rm1}
The following comments are in order.
\begin{itemize}
\item[1.] If we do not assume $\rho$ to be sufficiently large, then the uniqueness of solutions might fail. For example (when $d=1$) both $v\equiv 0$ and $v=x$ are solutions to
$
v-xv_x=0
$,
and both $v\equiv 0$ and $v=x^2+1$ are solutions to
$
v-\frac{1}{2}(1+x^2)v_{xx}=0
$. 

\item[2.] For any fixed large $\rho=N(N-1)$ with $N\geq 2$ an integer, the uniqueness of solutions still fails in general. To see this, we first construct the following $N+1$ numbers: $a_N=1$, $a_{N-1}=\frac{N}{2(N-1)}$, and define iteratively for $k=N-2,\ldots, 0$,
\[
a_k=\frac{(k+1)a_{k+1}+(k+2)(k+1)a_{k+2}}{N(N-1)-k(k-1)}.
\]
Then one can check directly that 
$
v=\sum_{k=0}^N a_kx^k$ and $v\equiv 0$
are both solutions to $\rho v-v_x-\frac12(1+x^2)v_{xx}=0$. This does not contradict Proposition \ref{P.4.1} because our result claims the unique solution among functions that grow as fast as a polynomial of power $c\sqrt{\rho}$ for some possibly small $c>0$ (depending on the assumption \eqref{c.3}). However, in the example, $\sqrt{\rho}= \sqrt{N(N-1)}$. 

Also, we cannot allow exponential growth of solutions (otherwise uniqueness fails). For instance, both $0$ and $e^x$ are solutions to $v-v_{xx}=0$.
These examples show the optimality of our assumptions on $\rho$ in terms of $N$ in Lemma \ref{L.cp} and Proposition \ref{P.4.1}.

\item[3.] If $b(\cdot,u)$ and $\sigma(\cdot, u)$ are only allowed to have a sublinear growth in $x$, that is,
\[
\lim_{R\to\infty}\sup_{(x,u)\in B_R\times U}\frac{|b(x,u)|+ |\sigma(x,u)|}{R}=0,
\]
then the existence and uniqueness of solutions, and the comparison principle hold the same without having to assume $\rho$ to be sufficiently large. The proof is similar to the one presented in the paper. We also refer the reader to \cite[Theorem 8]{tang2022exploratory}.
\end{itemize}
\end{Remark}

\section{Convergence of PIA with unbounded coefficients}\label{subsec:unbound.convergence}

This section concerns the case when $\sigma$ is independent of the control, and the coefficients can be unbounded.

Let us start with the following interior $W^{2,p}$ estimate. The classical result can be found, for example, in \cite[Chapter 3, Theorem 4.2]{chen1998second}. However, it is not sufficient for us as we need to carefully track the dependence of the constants on the size of the coefficients of the equation.

Let $\tilde v$ be a solution to
\be\label{eq:lm:unboundest} 
\rho \tilde v(x)- \tilde{r}(x) - \tilde{b}(x) \cdot {D}  \tilde v(x)-2^{-1}\tr\left( \tilde\Sigma (x) {D} ^2 \tilde v(x) \right) =0.
\ee 
\begin{Lemma}\lb{L.6.1}
Assume that $\rho\geq 1$, and  for some $ \tilde C_0>0$, $\tilde\Sigma=\tilde\sigma \tilde\sigma^T\geq \I_d/  \tilde  C_0$ and $\tilde\sigma$ is Lipschitz continuous with constant $\tilde  C_0$. Consider the ball $B_2(x_0)$ for an arbitrary $x_0\in \R^d$. Suppose there exist $\tilde A_1,\tilde A_2,\tilde A_3\geq 1$
such that 
\beq\lb{6.1}
\|\tilde r(\cdot)\|_{L^\infty(B_2(x_0))}\leq \tilde A_1,\,\,\|\tilde b(\cdot)\|_{L^\infty(B_2(x_0))}\leq \tilde A_2 ,\,\,\|\tilde\sigma(\cdot)\|_{L^\infty(B_2(x_0))}\leq \tilde A_3.
\eeq
Then for any $p>d$, there exists $C=C(\tilde C_0,p, d)>0$, which is independent of $\tilde A_1,\tilde A_2,\tilde A_3$, $\rho$, and $x_0$, such that the solution $\tilde v$ to \eqref{eq:lm:unboundest} satisfies
\beq\lb{6.2}
\|\tilde v\|_{W^{2,p}(B_1(x_0))}\leq C\left[
\tilde{A}_1 +(\tilde{A}_2^{2}+\tilde{A}_3^{{4+\frac{2d}{p}}}+\rho)\|\tilde  v\|_{L^\infty(B_{2}(x_0))}    \right].
\eeq
\end{Lemma}

\begin{proof}
Without loss of generality, we only prove the result for $x_0=0$. Below, we divide the proof into three steps.
	
{\bf Step 1.}  We first derive a Sobolev inequality over $B_R$ with explicit dependence of constants on $R\in (0,\frac{1}{2})$. Since $p>d$, we can apply Gagliardo-Nirenberg interpolation inequality to $\mu(x):=\tilde v(Rx)$ in $B_1$ to get
\[
\|D\mu\|_{L^p(B_1)}\leq C\|D^2 \mu\|_{L^p(B_1)}^\theta\|\mu\|_{L^\infty(B_1)}^{1-\theta}+C\|\mu\|_{L^\infty(B_1)},
\]
where the parameters satisfy
\[
\frac1p=\frac1d+\theta\left(\frac1p-\frac2d\right),
\]
and $C$ only depends on $p$ and $d$. By Young's inequality, we have for any $\eps\in (0,1]$  that
\[
\|D\mu\|_{L^p(B_1)}\leq \eps \|D^2 \mu\|_{L^p(B_1)}+C \eps^{-\frac{\theta}{1-\theta}}\|\mu\|_{L^\infty(B_1)},
\]
which implies, for $C$ only depending on $p, d$ and independent of $R, \eps$ that
\beq\lb{6.4}
\|D \tilde v\|_{L^p(B_R)}\leq  \eps R\|D^2 \tilde v\|_{L^p(B_R)}+ C \eps^{-1+\frac{d}p} R^{-1+\frac{d}p} \| \tilde v\|_{L^\infty(B_R)}.
\eeq
	
{\bf Step 2.}  In this step, we derive a uniform Sobolev estimate for a transformed function $\tilde w(x)$ defined below on small balls.

First, since $\tilde \Sigma=\tilde\sigma\tilde\sigma^T$ and \eqref{6.1}, there exists $c\in (0,1)$ depending on $d$ such that 
\be\label{eq:lm6.1.ball} 
	\left\{ \tilde\Sigma(x_1)^{1/2} x+x_1\,|\,x\in B_{c/\tilde A_3} \right\}\subset B_2\quad \text{ for all } x_1\in B_1.
\ee 
Set $\tilde R:= c'(\tilde A_2+\tilde{A}_3^2 )^{-1}$ for some $c'\in (0,c)$ to be determined in the proof.
Then, by the assumption that $\tilde\sigma$ is Lipschitz continuous, there exists $C'$ depending only on ${\tilde C_0}$ such that for any $0<R< \tilde R$,
\be\label{eq:lm6.1.Sigmay}  
\left| \tilde\Sigma\left( \tilde\Sigma(x_1)^{1/2} x+x_1 \right)-\tilde\Sigma(x_1) \right| \leq C'\tilde{A}_3^2 |x|\leq  C'\tilde{A}_3^2 R.
\ee

From now on, fix $x_1\in B_1$ and define $\tilde w(x):=\tilde v\left(\tilde\Sigma(x_1)^{1/2} x+x_1\right)$. \eqref{eq:lm6.1.ball} yields that $\tilde w$ is well-defined in a neighbourhood of $B_{\tilde R}$. A direct calculation shows that $\tilde w$ satisfies
	$
	-\frac{1}{2} \Delta \tilde w(x)= g(x)
	$
	with 
	\be\label{eq:lm6.1.g} 
	\begin{aligned}
		g(x):=&\tilde r\left( \tilde\Sigma(x_1)^{1/2} x+x_1 \right)+\tilde b\left( \tilde\Sigma(x_1)^{1/2} x+x_1 \right)\cdot \tilde\Sigma(x_1)^{-1/2}D\tilde w(x)\\
		& +2^{-1}\tr\left( \left( \tilde\Sigma(\tilde\Sigma(x_1)^{1/2} x+x_1)\tilde\Sigma(x_1)^{-1}-\I_d \right) D^2 \tilde w(x) \right)
		-\rho \tilde w(x).
	\end{aligned}
	\ee 

Next, for any $R\in(0, \tilde R]$ and $\eta \in (0,R)$, take $\zeta\in \Cc_0^\infty(B_{R})$ a cutoff function such that
\[
	\zeta\leq 1\,\text{ in }B_{R},\quad \zeta=1\,\text{ in }B_\eta, \quad\text{ and }\quad |D^k\zeta|\leq C (R-\eta )^{-k}\text{ with }k=1,2.
\]
Here the constant $C$ can be taken independent of $R$ and $\eta $.
Then $\hat{w}(x):=\tilde w ( x)\zeta(x)$ satisfies 
\[
2^{-1}\Delta \hat w(x)=g(x)\zeta(x)-D\tilde w(x)\cdot D\zeta(x) -2^{-1}\tilde w(x)\Delta \zeta (x)=:\hat g (x).
\]
By classical results for elliptic equations (e.g., \cite[Chapter 3, Theorem 3.6]{chen1998second}), we have that
\bee
\|D^2\hat w\|_{L^p(B_{ R})}\leq C\| \hat g \|_{L^p(B_{ R})},
\eee
where $C$ depends only on $p,d$. 
Then by the definition of $\zeta$, the above estimate implies that
\beq\lb{6.7}
\|D^2 \tilde w\|_{L^p(B_\eta)}\leq C\|g\|_{L^p(B_{R})}+C(R-\eta )^{-1}\|D\tilde w\|_{L^p(B_{R})}+C(R-\eta )^{-2}\|\tilde w\|_{L^\infty (B_{R})},
\eeq
with $C$ independent of $0<\eta<R\leq \tilde R$.
Then \eqref{eq:lm6.1.Sigmay}--\eqref{6.7} together yield that
\be\label{eq:lm6.1.cut}
\begin{aligned}
\|D^2\tilde w\|_{L^p(B_\eta)}
&\leq C \left(
		\|\tilde r\|_{L^\infty(B_1)}|B_{R}|^{1/p}+\|\tilde b\|_{L^\infty(B_1)} \|D\tilde w\|_{L^p(B_{R})}+\tilde{A}_3^2 R\|D^2\tilde w\|_{L^p(B_{R})}+\rho\|\tilde w\|_{L^p(B_{R})}
		\right)\\
		&\quad +C(R-\eta )^{-1}\|D \tilde w\|_{L^p(B_{R})}+C(R-\eta )^{-2}\|\tilde w\|_{L^\infty(B_{R})}\\
		&\leq C \left(
		\tilde A_1 R^{d/p}+\tilde  A_2 \|D \tilde w\|_{L^p(B_{R })}+\tilde{A}_3^2 R\|D^2 \tilde w\|_{L^p(B_{R})}+\rho R^{d/p}\| \tilde w\|_{L^\infty(B_{R })}
		\right)\\
		&\quad +C(R-\eta )^{-1}\|D\tilde w\|_{L^p(B_{R})}+C(R-\eta )^{-2}\|\tilde w\|_{L^\infty(B_{R})}.
\end{aligned}
\ee
where $C$ only depends on ${\tilde C_0}, p$, and $d$. In particular, it is independent of $0<\eta<R\leq \tilde R$ and $x_1\in B_1$. 

	
Now, it follows from \eqref{6.4} that there exists $C$ independent of $R\in (0,\tilde R)$ and $\eps\in (0,1]$ such that
\[
\|D\tilde w\|_{L^p(B_{R})}\leq  \eps R\|D^2 \tilde w\|_{L^p(B_{R})}+C \eps^{-1+\frac{d}{p}} R^{-1+\frac{d}{p}}\|\tilde w\|_{L^\infty(B_{R})}.
\]
Since $R\leq \tilde R\leq (\tilde A_2+\tilde{A}_3^2 )^{-1}$,
applying the above estimate with $\eps=1$ in the last but one line of \eqref{eq:lm6.1.cut} and with $\eps=(\tilde A_2+\tilde{A}_3^2 )(R-\eta)\leq 1$ in the last line of \eqref{eq:lm6.1.cut} yields
\begin{align*}
		\|D^2 \tilde w\|_{L^p(B_{\eta })}
		&\leq C_1(\tilde A_2+\tilde{A}_3^2 )R\|D^2 \tilde w\|_{L^p(B_{R })}+ C_1\left(
		\tilde A_1 R^{\frac{d}{p}}+(\tilde A_2 R^{-1+\frac{d}{p}}+\rho R^{\frac{d}{p}} )  \| \tilde w\|_{L^\infty(B_{R })}
		\right)\\
		&\quad +C_1(1+((\tilde A_2+\tilde{A}_3^2 )R)^{-1+\frac{d}{p}})(R-\eta )^{-2} \|\tilde w\|_{L^\infty(B_{R})}, 
\end{align*}
where $C_1\geq 1$ depends only on ${\tilde C_0}, p$, and $d$. 

Let us now take $c'=\min\{c,1/(2C_1)\}$, $\tilde R=c'(\tilde A_2+\tilde{A}_3^2 )^{-1}$, and $R_0:=\tilde R/2$.
Thus, for all  $R_0\leq \eta<R\leq \tilde R$, the above inequality yields for some $C_1'>0$ depending only on $C_1,c',d,p$,
\begin{align*}
		\|D^2 \tilde w\|_{L^p(B_{\eta })}
		&\leq \frac{1}{2}\|D^2 \tilde w\|_{L^p(B_{R })}+ C_1\left(
		\tilde A_1\tilde  R^{\frac{d}{p}}+(\tilde A_2 \tilde  R^{-1+\frac{d}{p}}+\rho \tilde  R^{\frac{d}{p}})  \| \tilde w\|_{L^\infty(B_{\tilde R})}
		\right)\\
		&\quad +C_1'(R-\eta )^{-2} \|\tilde w\|_{L^\infty(B_{\tilde R})}.
\end{align*}
We use \cite[Chapter 2, Lemma 4.1]{chen1998second} to get for all $R_0\leq \eta<R\leq \tilde R$ that
\[
\|D^2\tilde w\|_{L^p(B_{\eta})}\leq   C\left(
		\tilde A_1 R_0^{\frac{d}{p}}+(\tilde A_2 R_0^{-1+\frac{d}{p}}+\rho R_0^{\frac{d}{p}})  \| \tilde w\|_{L^\infty(B_{2R_0 })}
		\right) +C(R-\eta )^{-2} \|\tilde w\|_{L^\infty(B_{2R_0})},
\]
which yields
\be\label{eq:lm6.1.estw} 
\begin{aligned}
\|D^2\tilde w\|_{L^p(B_{R_0})}\leq   C\left(
		\tilde A_1 R_0^{\frac{d}{p}}+(\tilde A_2 R_0^{-1+\frac{d}{p}}+\rho R_0^{\frac{d}{p}}+R_0^{-2})  \| \tilde w\|_{L^\infty(B_{2R_0 })}
		\right) 
\end{aligned}
\ee
where $C$ only depends on ${\tilde C_0}, p, d$. For the reader's convenience, we copy \cite[Chapter 2, Lemma 4.1]{chen1998second} after the proof.

{\bf Step 3.} Let us turn back to $\tilde v$. 
Notice that there exists $\tilde C\geq 1$ only depending on ${C_0}$ such that $\tilde\Sigma(x_1)^{-1/2}B_R:=\{\tilde\Sigma(x_1)^{-1/2} x\,|\,x\in B_R\}\subset B_{\tilde CR}$ for all $x_1\in B_1$. Set $r_0=R_0/\tilde C$, and then
\eqref{eq:lm6.1.estw} and $\tilde w(x)=\tilde v(\tilde\Sigma(x_1)^{1/2} x+x_1)$ together give that
\beq\lb{6.8}
\begin{aligned}
		\|D^2\tilde  v\|_{L^p(B_{r_0}(x_1))}&=\left(\int_{\tilde\Sigma(x_1)^{-1/2} B_{r_0}}  |D^2\tilde w(z)|^p \tilde\Sigma(x_1)^{-1/2} dz\right)^{1/p}\leq C\|D^2\tilde w\|_{L^p\left( B_{R_0}\right)} \\
		&\leq C\left(
		\tilde{A}_1 r_0^{\frac{d}{p}}+(\tilde{A}_2r_0^{-1+\frac{d}{p}}+\rho r_0^{\frac{d}{p}} + r_0^{-2}  )\|\tilde  v \|_{L^\infty\left( B_2 \right)}\right),
\end{aligned}
\eeq
where we also used \eqref{eq:lm6.1.ball} in the last inequality, and the constant $C$ only depends on ${\tilde C_0}, p, d$.

Next, we can take $N:=(\lfloor\sqrt{d}/ r_0 \rfloor+1)^d$ balls centered at $\{y_1,y_2,\ldots,y_N\}\subseteq B_1$ such that $B_1\subseteq \bigcup_{i=1}^N B_{r_0}(y_i)$. By applying \eqref{6.8} to $x_1=y_1,\cdots, y_N$, we have that
	\begin{align*}
		\|D^2\tilde v\|_{L^p(B_1)}&\leq \left(\sum_{i=1}^N \|D^2\tilde v\|_{L^p(B_{r_0}(y_i))}^p\right)^{1/p}
		\leq  CN^{1/p}\left[
		\tilde{A}_1 r_0^{\frac{d}{p}}+(\tilde{A}_2r_0^{-1+\frac{d}{p}}+\rho r_0^{\frac{d}{p}} + r_0^{-2}  )\|\tilde  v \|_{L^\infty\left( B_2 \right)}\right]\\
		& \leq C\left[
		\tilde{A}_1+( \tilde{A}_2r_0^{-1}+\rho + r_0^{-2-\frac{d}{p}})\| \tilde v\|_{L^\infty(B_{2})}\right].
	\end{align*}
Recall that $r_0=c'/\tilde C(\tilde A_2+\tilde{A}_3^2)^{-1}$.
By \eqref{6.4} again, we obtain
\[
\|\tilde v\|_{W^{2,p}(B_1)}\leq C\left[
\tilde{A}_1 +(\tilde{A}_2^{2}+\tilde{A}_2\tilde{A}_3^{2}+\tilde{A}_3^{{4+\frac{2d}{p}}}+\rho)\| \tilde v\|_{L^\infty(B_{2})}    \right]
\]
where $C$ is independent of $\tilde{A}_1,\tilde{A}_2,\tilde{A}_3$, and $\rho$. This finishes the proof.
\end{proof}

Let us state Lemma 4.1 from \cite[Chapter 2]{chen1998second} that was used in the proof of Lemma \ref{L.6.1}.

\begin{Lemma}
Let $\varphi(r)$ be a bounded nonnegative function defined on the interval $[R_0,R_1]$, where $R_1>R_0\geq 0$. Suppose that for any $R_0\leq \eta<R\leq R_1$, $\varphi$ satisfies
\[
\varphi(\eta)\leq \theta\varphi(R)+\frac{A}{(R-\eta)^\alpha}+B
\]
where $\theta,A,B$, and $\alpha$ are nonnegative constants, and $\theta<1$. Then
\[
\varphi(\eta)\leq C\left[\frac{A}{(R-\eta)^\alpha}+B\right],\quad\text{ for all }R_0\leq \eta<R\leq R_1,
\]
where $C$ depends only on $\alpha,\theta$.
\end{Lemma}

Now we study PIA. 
We first assume
\beq\lb{H1}
\begin{cases}
\eqref{c.1}, \eqref{c.3},\eqref{c.4},\text{ and }\Sigma(x,u)=\Sigma(x)\geq \I_d/C_0,\text{ which is independent of }u,    \\
\text{the maps $u\to r(x,u)$ and $u\to b(x,u)$ are uniformly Lipschitz continuous.}
\end{cases}
\eeq
Next, suppose that $v^0$ is locally uniformly $\Cc^{1,\alpha}$ for some $\alpha\in(0,1)$, and 
\beq\lb{H2}
\sup_{R\geq 1}R^{-N}\|v^0\|_{\Cc^{1,\alpha}(B_R)}<\infty.
\eeq
And  for some $L\geq 1$,
\beq\lb{H3}
U=[0,1]^L.
\eeq
We comment that \eqref{H3} can be generalized to a uniform cone test condition as \cite[Assumption 4.2]{pia2022}: for any $u\in U$, there must exist a common-sized cone with its vertex at $u$ that is entirely contained within $U$.
This condition is not restrictive, as it is satisfied when $U$ is a convex set or a finite union of convex sets. Moreover, this condition, combined with the Lipschitz condition on $u$, leads to the logarithmic growth estimate for the entropy term in \cite[Corollary 4.2]{pia2022}, which is essential in the proof of Lemma \ref{L.5.1} below.

\medskip

Since $\sigma$ is independent of $u$, we have a simpler formula for $\pi^n$. Indeed, for $n\geq 1$, define 
\beq\lb{5.3}
\pi^n(x)(u):=\Gamma(x, D v^{n-1} (x)) (u)=\frac{\exp\left(  \frac{1}{\lambda} (r(x,u)+ b(x,u) \cdot D v^{n-1}(x)\right)}{\int_U \exp\left(  \frac{1}{\lambda} (r(x,u')+ b(x,u') \cdot D v^{n-1}(x)\right) du'},
\eeq
and $r^n,b^n,\Hc^n$ the same as in \eqref{1.10} with the above $\pi^n$. We then look for $v^n$ from the equation:
\beq\lb{5.1}
	\rho v^n(x)- {r}^n(x) - {b}^n(x) \cdot {D}  v^n(x)-{2}^{-1}\tr({\Sigma}(x) {D} ^2 v^n(x)) +\lambda {\Hc}^n(x)=0.
\eeq
Below we show the well-posedness of \eqref{5.1} and thus the PIA, and we use the uniform ellipticity of the equation to obtain some uniform estimates on $v^n$ for all $n\geq 0$.

\begin{Lemma}\lb{L.5.1}
Under the assumptions of \eqref{H1}--\eqref{H3},
there exists $C>0$ independent of $\rho$ such that for all $n\geq 1$ and any $R>0$, 
\[
\rho \|v^n\|_{L^\infty(B_R)}\leq C(1+R^N) \quad\text{and}\quad
\|Dv^n\|_{\Cc^{\alpha}(B_R)}\leq C(1+|R|^{N+5}).
\]
\end{Lemma}

\begin{proof}

Let $M\geq 2$ be a large constant, so that it satisfies for some $\tilde C\geq 1$ to be determined (independent of $n,\rho$),
\[
\sup_{R\geq 1}R^{-N}\|v^0\|_{\Cc^{1,\alpha}(B_R)}\leq M\quad\text{and}\quad M\geq \tilde C(1+\ln M).
\]

Assume for induction that  for some $n\geq 1$, $v^{n-1}$ exists, and for all $x_0\in \R^d$ we have 
\beq\lb{5.5}
\|{D}  v^{n-1}\|_{\Cc^\alpha(B_1(x_0))}\leq M(1+|x_0|^{N+5}).
\eeq
Recall \eqref{1.10} and \eqref{5.3}. We get from \eqref{c.3} that
\beq\lb{5.6}
|r^n(x)|\leq \sup_{u\in U} |r(x,u)|\leq A_1(1+|x|^N)\quad\text{and}\quad |b^n(x)|\leq \sup_{u\in U} |b(x,u)|\leq A_2(1+|x|).
\eeq
Since $U=[0,1]^L$, the cone test condition in \cite[Corollary 4.2]{pia2022} is satisfied. Thus, the corollary yields that there exists $C$ depending on $L,\lambda,N$ and other constants in the assumptions such that
\bee
\begin{aligned}
|\lambda\Hc^n(x)|&=\left|\lambda\int_U\ln(\Gamma(x,Dv^{n-1}(x)))\Gamma(x,Dv^{n-1}(x))\,du\right|\\
&\leq C(1+\ln(1+|{D} v^{n-1}(x)|))\leq C(\ln M+\ln (1+|x|)) 
\end{aligned}
\eee
where, in the second inequality, we applied the induction hypothesis \eqref{5.5}.
Thus, 
\beq\lb{5.7}
|r^n(x)-\lambda\Hc^n(x)|\leq (A_1+2C\ln M)(1+|x|^N).
\eeq
Then it follows from \eqref{c.4} and Proposition \ref{P.4.1} (with $r^n(x)-\lambda\Hc^n(x), b^n(x)$ in place of $r(x,u), b(x,u)$) that
there exists a unique solution $v^n$ to \eqref{5.1} such that
\beq\lb{5.2}
\rho|v^n(x)|\leq C_1(A_1+\ln M) \left( 1+|x|^N \right)  \quad \text{ for all } x\in \R^d,
\eeq
where $C_1$ depends only on $L, \lambda ,N$, and the assumptions. 
By applying Lemma \ref{L.6.1} to $v^n$ with an arbitrary $p>d$, we have for any point $x_0\in\R^d$ that,
\begin{align*}
\| v^n\|_{W^{2,p}(B_2(x_0))} & \leq  C 
\Big[ \| r^n-\lambda {\Hc^n}  \|_{L^\infty(B_4(x_0))} +\\
&\quad \left(\rho+ \|b^n\|^2_{L^\infty(B_4(x_0))}+\|\sigma\|^{4+\frac{2d}p}_{L^\infty(B_4(x_0))} \right)\|v^n\|_{L^\infty(B_4(x_0))}   \Big],
\end{align*}
Recall $\rho\geq 4(N+1)(A_2+NA_3)$ from \eqref{c.4}. We apply \eqref{c.3}, \eqref{5.6}, \eqref{5.7}, and \eqref{5.2} to get
\be\label{eq3} 
\begin{aligned}
	\| v^n\|_{W^{2,p}(B_1(x_0))}\leq C\left(A_1+\ln M \right)\left( A_2+A_3^{3+\frac{2d}{p}}\right)\left(1+|x_0|^{N+4+\frac{2d}{p}}\right)
\end{aligned}
\ee 
with $C$ depending only on $C_0, p, d$ (independent of $\rho\geq 1$).

Now, taking $p=\max\{\frac{d}{1-\alpha},2d\}$ and using \eqref{eq3}, we apply the Sobolev embedding (e.g., \cite[Section 5.6.2, Theorem 5]{Evans-book-98}) to get
\be\label{eq:lm6.1.deriv} 
\|v^n\|_{\Cc^{1,\alpha}(B_1(x_0))}\leq C\| v^n\|_{W^{2,p} (B_2(x_0))}\leq C \left(A_1+\ln M \right)\left( A_2+A_3^{4}\right)\left(1+|x_0|^{N+5}\right),
\ee 
where $C$ is independent of $x_0,\rho$, and $n\geq 1$. 
Hence if $M$ is sufficiently large such that
\beq\lb{6.15}
C \left(A_1+\ln M \right)\left( A_2+A_3^{4}\right)\leq M,
\eeq
then
\[
\|{D}  v^n\|_{\Cc^\alpha(B_1(x_0))}\leq M(1+|x_0|^{N+5}).
\]
By induction, we finish the proof of \eqref{5.5}. 

Finally, \eqref{5.2} finishes the proof of the lemma.
\end{proof}

We note that under the same conditions as Lemma \ref{L.5.1}, the controlled SDE for each iteration step $n$, given by
$
dX^{\pi^n}_t=b^n(X^{\pi^n}_t)\,dt+\sigma(X^{\pi^n}_t)\,dW_t,$
admits a unique strong solution. This follows from the fact that $b^n(x)$ and $\sigma(x)$ are both locally Lipschitz and exhibit linear growth in $x$.

\begin{Corollary}\lb{C.5.1}
Under the assumptions of Lemma \ref{L.5.1}, then we have
\[
v^{n-1}\leq v^{n}\quad \text{ for all } n\geq 1.
\]
\end{Corollary}

\begin{proof}
By Proposition \ref{P.4.1} and Lemma \ref{L.5.1}
\[
 \rho \|v^n\|_{L^\infty(B_R)}\leq C(1+R^N) \quad \text{ for all } R>0,
\]
for some constant $C$ independent of $n$ and $\rho$, and $v^n$ is a subsolution for \eqref{4.0}. Notice that $v^{n-1}$ being a subsolution to  \eqref{5.1} for any $n\geq 1$.  Thus, with $r^n(x)-\lambda \Hc^n(x), b^n(x)$ in place of $r(x,u), b(x,u)$ in \eqref{4.1}, we apply Lemma \ref{L.cp} to \eqref{5.1} to get that $v^n\geq v^{n-1}$ for any $n\geq 1$.
\end{proof}

The next goal is to obtain the convergence of $v^n$ as $n\to\infty$. 

\begin{Theorem}\lb{T.5.1}
Under the assumptions \eqref{H1}--\eqref{H3}, let $v^*$ solve \eqref{2.1} with $\sigma$ independent of $u$, and let $v^n$ from PIA. Then $v^n\to v$ as $n\to\infty$ locally uniformly in $\Cc^{1,\alpha}$ over $\R^d$.
\end{Theorem}

\begin{proof}
By the uniform local bound of $\{v^n\}_n$ in Lemma \ref{L.cp}, Corollary \ref{C.5.1}, and the Monotone Convergence Theorem, $v^n$ converges locally uniformly to some function, denoted by $\bar{v}$.

By Lemma \ref{L.5.1}, $v^n$ and $\bar{v}$ are locally uniformly bounded in $\Cc^{1,\alpha}$. 
Therefore, we actually have that $v^n\to \bar{v}$ locally uniformly  in $\Cc^{1,\alpha}$. 
This implies that $\pi^n(x,u)\to \Gamma(x,D\bar v)(u)$ locally uniformly as $n\to\infty$.
By the definition of $r^n,b^n$, and the stability of viscosity solutions (under locally uniform convergence), we get that $\bar{v}$ is a viscosity 
to
\[
\rho v- \int_U \left[r(x,u)+ b(x,u) \cdot Dv+ \frac{1}{2}\tr(\sigma\sigma^T(x) D^2 v) -\lambda \ln(\Gamma(x,Dv)(u))\right] \Gamma(x,Dv)(u)\,du =0.
\]
The definition of $\Gamma$ then yields that $\bar{v}$ is a viscosity solution to \eqref{2.1}. Thus, by the comparison principle, $\bar{v}=v^*$, which finishes the proof.
\end{proof}

\bibliographystyle{abbrv}

\section*{Statements and Declarations}
H. Tran is partially supported by NSF CAREER grant DMS-1843320, a Simons Fellowship, and a Vilas Faculty Early-Career Investigator Award. Y. P. Zhang is partially supported by Simons Foundation Travel Support MPS-TSM-00007305, and a start-up grant at Auburn University.

\end{document}